\documentclass{amsart}
\usepackage[active]{srcltx}
\title[]{
Zero mean curvature entire graphs of
mixed type in Lorentz-Minkowski 3-space
}
\date{\today}
\usepackage[dvips]{graphicx}
\usepackage{enumerate}
\usepackage{amssymb}
\usepackage{eepic}
\usepackage[usenames]{color}
\newcommand{\red}[1]{\textcolor{Red}{#1}}
\renewcommand{\red}[1]{\textcolor{Black}{#1}}
\usepackage{amsthm}
\theoremstyle{plain}
 
 \newtheorem{theorem}{Theorem}[section]
 \newtheorem*{theorem*}{Theorem}

 \newtheorem*{lemma*}{Lemma}
 \newtheorem{proposition}[theorem]{Proposition}
 \newtheorem{fact}[theorem]{Fact}
 \newtheorem*{fact*}{Fact}

 \newtheorem{lemma}[theorem]{Lemma}
 \newtheorem{corollary}[theorem]{Corollary}
 \theoremstyle{remark}

 \newtheorem{definition}[theorem]{Definition}
 \newtheorem{remark}[theorem]{Remark}
 \newtheorem*{remark*}{Remark}
 \newtheorem*{problem*}{Problem}
 
 \newtheorem{example}[theorem]{Example}
 
\numberwithin{equation}{section}
\numberwithin{figure}{section}

\newcommand{\vect}[1]{\boldsymbol{#1}}

\newcommand{\Z}{\boldsymbol{Z}}

\newcommand{\R}{\boldsymbol{R}}
\newcommand{\C}{\boldsymbol{C}}

\newcommand{\imag}{\mathrm{i}}

\newcommand{\op}[1]{\operatorname{#1}}
\newcommand{\pmt}[1]{{\begin{pmatrix} #1  \end{pmatrix}}}

\renewcommand{\Re}{\operatorname{Re}}

\renewcommand{\phi}{\varphi}
\renewcommand{\epsilon}{\varepsilon}

\newcommand{\vmt}[1]{{\begin{vmatrix} #1  \end{vmatrix}}}

\author{S.~Fujimori}%{Shoichi Fujimori}
\address[Shoichi Fujimori]{%
   Department of Mathematics, Okayama University,
   Tsushima-naka, Okayama 700-8530, Japan}
\email{fujimori@math.okayama-u.ac.jp}

\author{Y. Kawakami}%{Yu Kawakami}
\address[Yu Kawakami]{%
Graduate School of Natural Science and Technology,
Kanazawa University,
Kanazawa, 920-1192, Japan,
}
\email{y-kwkami@se.kanazawa-u.ac.jp}

\author{M. Kokubu}%{Masatoshi Kokubu}
\address[Masatoshi Kokubu]{%
Department of Mathematics, School of Engineering, 
Tokyo Denki University, 
Tokyo 120-8551, Japan}
\email{kokubu@cck.dendai.ac.jp}

\author{W.~Rossman}%{Wayne Rossman}
\address[Wayne Rossman]{%
   Department of Mathematics, Faculty of Science,
   Kobe University,
   Rokko, Kobe 657-8501, Japan
}
\email{wayne@math.kobe-u.ac.jp}

\author{M.~Umehara}%{Masaaki Umehara}
\address[Umehara]{%
   Department of Mathematical and Computing Sciences,
   Tokyo Institute of Technology
   2-12-1-W8-34, O-okayama, Meguro-ku,
   Tokyo 152-8552, Japan.
}
\email{umehara@is.titech.ac.jp}

\author{K.~Yamada}%{Kotaro Yamada}
\address{%
   Department of Mathematics\\
   Tokyo Institute of Technology\\
   O-okayama, Meguro, Tokyo 152-8551\\
   Japan
}
\email{kotaro@math.titech.ac.jp}

\dedicatory{Dedicated to Professor 
Osamu Kobayashi for his sixtieth birthday.}
\thanks{
Fujimori was partially supported by the Grant-in-Aid for Young
Scientists (B) No. 25800047, Kawakami was supported  by the Grant-in-Aid
for Scientific Research (C) No.\ 15K04845, Rossman by Grant-in-Aid for
Scientific Research (C) No.\ 15K04845 , Umehara by (A) No.\ 26247005
and Yamada
by (C) No. 26400066 from Japan Society for the Promotion of Science.
}
\begin{document}
\begin{abstract}
It is classically known that the only zero mean 
curvature entire graphs in the Euclidean 3-space 
are planes, by Bernstein's theorem.
A surface in Lorentz-Minkowski 3-space $\R^3_1$
is called of {\it mixed type} if it changes causal type
from space-like to time-like.
In $\R^3_1$,
Osamu Kobayashi found two zero mean 
curvature entire graphs of mixed  
type that are not planes. 
As far as the authors know, 
these two examples were the only known examples
of entire zero mean curvature graphs of mixed 
type  without singularities. 
In this paper, we construct several families of 
real analytic zero mean curvature entire graphs of 
mixed type in Lorentz-Minkowski $3$-space.
The entire graphs mentioned above lie in one of 
these classes.
\end{abstract}
\maketitle

%%%%%%%%%%%%%%%%%%%%%%%%%%%%%%%%%%%%%%%%%%%%%%%%%%%%%%%%%%%%%%%%%%%%%
\section*{Introduction}
The Jorge-Meeks $n$-noid ($n\ge 2$)
is a complete minimal surface of genus zero 
with $n$ catenoidal  ends  in the Euclidean 
3-space $\R^3$, which has
 $(2\pi/n)$-rotation symmetry with respect to 
an axis. 
In the authors' previous  work \cite{FKKRUY}, 
it was shown that the corresponding space-like 
maximal 
surface $\mathcal J_n$
in Lorentz-Minkowski 3-space $(\R^3_1;t,x,y)$ 
has an analytic extension $\tilde {\mathcal J}_n$
as a properly embedded zero mean curvature 
surface which changes its causal type 
from space-like to time-like. 
Noting that a smooth function 
$\lambda:\R^2\to \R$ can be realized as a subset 
$$
\bigl\{(t,x,y)\in \R^3_1\,;\, t=\lambda(x,y),\,\, x,y \in \R\bigr\},
$$
which we call the graph of the function $\lambda(x,y)$,
it should be remarked that $\tilde {\mathcal J}_n$ 
for $n \ge 3$ cannot be expressed as a graph,
but $\tilde {\mathcal J}_2$ 
gives an  entire zero mean curvature graph
(cf. \cite[Section 1]{FKKRUY}) associated to
\begin{equation}\label{eq:J0}
\lambda(x,y):=x\tanh 2y.
\end{equation}

Remarkably, until now, only two
entire graphs  of mixed type 
with zero mean curvature
were known.
One  is  \eqref{eq:J0} as above, and
the other is the Scherk type entire 
zero mean curvature graph
 associated to
\begin{equation}\label{eq:K0}
\lambda(x,y):=\log(\cosh x/\cosh y).
\end{equation}
Both surfaces were found  by Osamu Kobayashi \cite{K}. 
Recently,  Sergienko and Tkachev
\cite{ST} produced several interesting 
entire zero mean curvature graphs which admit
cone-like singularities (the surface
given by \eqref{eq:K0} is contained amongst
their examples). 

Here, the condition \lq mixed type\rq\
for zero mean curvature entire graphs is important.
In fact, Calabi \cite{C}
proved that there are no space-like zero 
mean curvature entire graphs except for planes.
On the other hand,  there are many time-like zero mean 
curvature entire graphs. For example,  the graph of 
$t=y+\mu(x)$ gives such an example if $\mu'(x)>0$ for 
all $x\in \R$. It should also be  remarked that
there are no non-zero constant mean curvature 
surfaces of mixed type (see \cite{HKKUY}).

The purpose of this paper is to construct 
further examples of entire zero mean 
curvature graphs of mixed type.
In fact,  we give a real $(4n-7)$-parameter 
(resp. $3$-parameter)
family $\mathcal K_{n}$
of space-like maximal surfaces of genus zero which 
admit only fold-type singularities
for each integer $n\ge 3$ (resp. for $n=2$), up to
motions and a homothetic transformations in $\R^3_1$.
The Jorge-Meeks type maximal surface $\mathcal J_n$
mentioned above belongs to $\mathcal K_n$.
Also,  the previously known two entire graphs 
of mixed type given in Kobayashi \cite{K}
are contained in $\mathcal K_{2}$,
so we call the surfaces in the class
$\mathcal K_{n}$ the \lq {\it Kobayashi surfaces}\rq\ 
of order $n$. 
Applying the technique used in \cite{FKKRUY},
we show that some (open)
subset 
of $\mathcal K_n$ 
consists of space-like maximal surfaces
which have analytic extensions across their
fold singularities, as proper immersions.
Moreover,  in this subclass of $\mathcal K_n$,
we can find a $(4n-7)$-parameter family of 
zero mean curvature entire graphs of
mixed type, 
up to congruence and homothety
in $\R^3_1$.
As a consequence, the maximal surfaces
$\mathcal S_n$ $(n\ge 2)$
which correspond to 
Scherk-type saddle towers in the Euclidean 
$3$-space
constructed by Karcher \cite{Kar}
(see also \cite{PT}) can be analytically extended 
as entire graphs in $\R^3_1$.
Kobayashi's entire graph
\eqref{eq:K0} coincides with $\mathcal S_2$.

During the production of this
paper, the authors received a 
preprint of Akamine \cite{A}, where
the entire graph as in Example \ref{angle:third}
is obtained by using a different
approach.
This surface belongs to
$\mathcal K_2$, as well as
the entire graphs (0.1) and (0.2).
As a consequence of Akamine's
construction, we can observe that
the surface is foliated by parabolas.

\section{Preliminaries}\label{sec:prelim}
We denote by $\R^3_1$
the Lorentz-Minkowski 3-space of
signature $(-++)$.
We introduce the notion of maxface
given in \cite{UY}:

\begin{definition}
Let $M^2$ be a Riemann surface. 
A $C^\infty$-map $f:M^2\to \R^3_1$
is called a {\it generalized maximal surface} 
(cf. \cite{ER}) if there exists an open dense
subset $W$ of $M^2$ such that the restriction
$f|_W$ of $f$ to $W$ gives a conformal (space-like) 
immersion of zero mean curvature. 
A {\it singular point} of $f$ is a
point (on $M^2$) at 
which $f$ is not an immersion. 
A singular point $p$ satisfying $df(p) = 0$
is called a {\it branch point} of $f$. 
A generalized maximal surface $f$ is called a {\it maxface} 
if $f$ does not have any branch points. 
\end{definition}

A maxface may have singular points in general.
Generic singularities of maxface
are classified in  \cite{FSUY}.
As noted in \cite[Remark 2.8]{FKKRSUYY2},
a maxface induces a holomorphic immersion
$
F=(F_0,F_1,F_2):\tilde M^2 \to \C^3
$
satisfying the nullity condition, that is, it satisfies
$$
-(dF_0)^2+(dF_1)^2+(dF_2)^2=0,
$$
where $\tilde M^2$ is the universal covering
of $M^2$. This immersion $F$ is called
the {\it holomorphic lift} of $f$.
We set
\begin{equation}\label{eq:W0}
g:=-\frac{dF_0}{dF_1-\imag dF_2},\qquad 
\omega:=\frac12(dF_1-\imag dF_2),
\end{equation}
where $\imag=\sqrt{-1}$. Then $(g,\omega)$ is a pair of
a meromorphic function and a meromorphic 
1-form on $M^2$, and is called 
the {\it Weierstrass data} of $f$.
The meromorphic function $g$ 
can be identified with 
the Gauss map of $f$.
Using $(g,\omega)$, 
$f$ has the following expression
(cf. \cite{K} and \cite{UY})
$$
f=\mbox{Re}\int_{z_0}^z(-2g, 1+g^2,\imag (1-g^2))\omega.
$$
The first fundamental form of $f$ is given by
\begin{equation}\label{eq:dsL}
ds^2=(1-|g|^2)^2 |\omega|^2.
\end{equation}
In particular, the singular set of $f$ consists of 
the points where $|g|=1$. 
In this situation, we set
\begin{align}\label{eq:F}
f_E&:=\op{Re}(F_1,F_2,\imag F_0)
=\mbox{Re}\int_{z_0}^z
(1+g^2,\imag (1-g^2),-2\imag g)\omega \\
&=\mbox{Re}\int_{z_0}^z(1-(-\imag g)^2,\imag 
(1+(-\imag g)^2),2(-\imag g))\omega.
\nonumber
\end{align}
This map $f_E$ gives a conformal minimal immersion
called the {\it companion} of $f$ (cf. \cite{UY}).
The Weierstrass data of $f_E$ is 
$(-\imag g,\omega)$.
Even if $f$ is single-valued on $M^2$,
its companion $f_E$ is defined only on 
$\tilde M^2$ in general.
The first fundamental form of $f_E$
is given by
\begin{equation}\label{eq:dsE}
ds^2_E:=(1+|g|^2)^2 |\omega|^2.
\end{equation}
A maxface $f$ is called {\it weakly complete}
if $ds^2_E$ 
in \eqref{eq:dsE}
is a complete Riemannian metric
on $M^2$ (cf. \cite[Definition 4.4]{UY}).
The definition of 
completeness of maxface is also
given in \cite[Definition 4.4]{UY}.
A complete maxface is weakly complete.
(The relationships between the two 
types of completeness
are written in \cite[Remark 6]{UY2}.)

For example, the maxface given by
\begin{equation}\label{eq:elliptic}
f(u,v)=(u,\cos v \sinh u,\sin v \sinh u)
\qquad
(u\in \R,\,\, v\in [0,2\pi)),
\end{equation}
is called the {\it elliptic catenoid}.
The surface is rotationally symmetric 
with respect to the time axis,
and has a cone-like singularity at the
origin.
If we set 
$z=r e^{\imag v}$ and $u:=\log r$,
then the map $f$ is the real part of the
holomorphic immersion
$$
F(z)=\left(\log z,\frac{z-z^{-1}}2,
-\imag \frac{z+z^{-1}}2\right).
$$
The companion of $f$ is the helicoid in the
Euclidean 3-space. The Weierstrass data
of $f$ is given by $(g,\omega)=(-z,dz/(2z^2))$.
Several maxfaces defined on non-zero genus
Riemann surfaces are known
(cf. \cite{FRUYY1}).
In this paper, we are interested in maximal surfaces
having only fold singularities defined as follows:

\begin{definition}\label{def:fold}
Let $M^2$ be a Riemann surface,
and $(g,\omega)$ the Weierstrass data
of a weakly complete
maxface $f:M^2\to \R^3_1$.
A singular point $p\in M^2$ is called {\it non-degenerate}
if $dg(p) \ne 0$ holds.
By the implicit function theorem,
there exists a regular curve $\sigma(t)$
($|t|<\epsilon$) on $M^2$ for some $\epsilon>0$
parameterizing the singular set of $f$
such that $\sigma(0)=p$.
A non-degenerate singular point $p$ 
is called a {\it fold} singularity  if it satisfies
\begin{equation}\label{cf}
\mbox{Re}\left(
\frac{dg(\sigma(t))}{g^2(\sigma(t))\omega(\sigma(t))}
\right)=0
\qquad (|t|<\epsilon).
\end{equation}
\end{definition}

The following fact 
follows from  
\cite[Theorem 2.15 and Lemma 2.17]{FKKRSUYY2}:

\begin{fact}
For each fold singular point $p$ of
a maxface $f$, there exists a local coordinate 
system $(U;u,v)$ centered at $p$ such that 
$f(u,v)=f(u,-v)$ and the $u$-axis consists of
the fold singularities
of $f$.  Moreover, the image of $f$ has an analytic extension
across the image of the singular
curve $u\mapsto f(u,0)$.
\end{fact}

In this paper, we are interested in  maxfaces
having only fold singularities on 
the Riemann surface whose compactification
is bi-holomorphic to the Riemann sphere
$$
S^2:=\C\cup \{\infty\}.
$$
\begin{definition}\label{def:W}
A pair $(g,\omega)$ 
consisting of a meromorphic function and a 
meromorphic $1$-form
defined on $S^2$ is called a {\it Weierstrass data on 
$S^2$} if the metric
\begin{equation}\label{eq:dse}
ds^2_E:=(1+|g|^2)^2 |\omega|^2
\end{equation}
has no zeros on $S^2$.
Each point where $ds^2_E$ diverges
is called an {\it end} of 
$(g,\omega)$, that is,
at least one of the three $1$-forms
\begin{equation}\label{three}
\omega,\quad g\omega,\quad g^2\omega
\end{equation}
have poles at the ends.
\end{definition}

We fix a Weierstrass data $(g,\omega)$
on $S^2$. Let $\{p_1,...,p_N\}$ be 
the set of ends of $(g,\omega)$. Then the map
\begin{equation}\label{eq:FL}
f=\mbox{Re}(F),\quad
F:=
\int_{z_0}^z(-2g, 1+g^2,\imag (1-g^2))\omega
\end{equation}
is defined on the universal covering of 
$S^2\setminus\{p_1,...,p_N\}$.
We call $f$ the maxface
{\it associated to}  $(g,\omega)$.
If $f$ is single-valued on
$S^2\setminus\{p_1,...,p_N\}$,
that is, the residues 
of the three meromorphic
1-forms
$-2g\omega, (1+g^2)\omega,
\imag (1-g^2)\omega$
are all real numbers
at each $p_j$ ($j=1,...,N$),
then we say that
$(g,\omega)$ 
{\it satisfies the period condition}.

\begin{proposition}\label{prop:wq}
Let $(g,\omega)$ be a Weierstrass data 
on $S^2$, and $p_1,...,p_N$
its ends. If the maxface $f$
associated to  $(g,\omega)$
is single-valued on
$S^2\setminus \{p_1,...,p_N\}$,
then $f:S^2\setminus \{p_1,...,p_N\}\to \R^3_1$
is a weakly complete maxface.
\end{proposition}

\begin{proof}
The map $F$ given in \eqref{eq:FL}
is an immersion if and only if
the metric $ds^2_E$ given by
\eqref{eq:dse} is positive definite.
Thus $f$ gives a maxface
defined on $S^2\setminus\{p_1,...,p_N\}$.
Since $ds^2_E$ diverges at $\{p_1,...,p_N\}$, 
at least one of the three 1-forms
as in \eqref{three} has a pole at $z=p_j$
($j=1,...,N$).
We let $m_j(\ge 1)$ be the maximum order 
of the poles of 
the above three forms at $p_j$.
Using this, one can easily
show that 
$
{|z-p_j|^{2m_j}}{ds^2_E}
$
is positive definite on a sufficiently small 
neighborhood of $z=p_j$.
This implies that $ds^2_E$ is 
a complete Riemannian
metric on $S^2\setminus \{p_1,...,p_N\}$.
\end{proof}

For example, the ends of the Weierstrass data
$
(g,\omega)=(z,{dz}/{(2z^2)})
$
of the elliptic catenoid
(cf. \eqref{eq:elliptic})
consist of $\{0,\infty\}$.
We are interested in a special class of
weakly complete maxface on $(M^2=)S^2$
as follows:

\begin{definition}\label{def:f-data}
The Weierstrass data $(g,\omega)$
on $S^2$ is called of {\it fold-type} if
\begin{enumerate}
\item[{(i)}] \label{i:W1}
all of its ends $p_1,...,p_N$ lie in
the unit circle 
$
S^1:=\{z\in S^2\,;\, |z|=1\},
$ 
\item[{(ii)}]  \label{i:W2}
$|g(z)|=1$ holds if and only if $z\in S^1$,
\item[{(iii)}]  \label{i:W3}
$
\mbox{Re}\left[{dg}/(g^2\omega)\right]
$
vanishes identically on $S^1$.
\end{enumerate}
\end{definition}

A  $C^\infty$-map
$\phi:M^2\to \R^3$ has a 
{\it fold singularity} at $p$ 
if there exists a local coordinate 
system $(u, v)$ centered at $p$ such that
$\phi(u, v) =\phi(u,-v)$. 
As shown in \cite[Lemma 2.17]{FKKRSUYY2},
the maxface induced by
a Weierstrass data 
of fold-type actually has fold singularities
along $S^1$.

\begin{example}[Scherk-type entire graph]\label{ex:Sch}
A typical example of
Weierstrass data of fold-type on $S^2$ is
$$
g=z,\qquad \omega=\frac{2dz}{z^4-1}.
$$
The ends of the Weierstrass data of $(g,\omega)$ 
consist of the roots $\{\pm1, \pm \imag\}$ of the equation $z^4=1$.
One can easily check that $(g,\omega)$ is 
of fold-type, and induces a maxface
$$
f(z)=\left(\log\left|\frac{1+z^2}{1-z^2}\right|,
\log\left|\frac{1-z}{1+z}\right|,
\log\left|\frac{1-\imag z}{1+\imag z}\right|\right).
$$
By setting $z=r e^{\imag \theta}$, $f=(t,x,y)$
satisfies
$$
\cosh x=
\frac{1+r^2}{|1-z^2|},\quad
\cosh y=
\frac{1+r^2}{|1+z^2|}.
$$
Thus we have
$$
\frac{\cosh x}{\cosh y}
=\left|\frac{1+z^2}{1-z^2}\right|=e^{t},
$$
that is, the image of $f$ satisfies 
the relation \eqref{eq:K0} in the introduction.
We remark that the conjugate $f_*$ of
$f$ induces the 
triply periodic 
maximal surface
given in \cite[Remark 2.4]{FRUYY2} with $\theta=0$,
which is obtained
as a limit of the family of
Schwarz P-type maximal
surfaces.
The fundamental piece of
$f_*$ is 
bounded by four light-like line segments.
\end{example}

We next give an example of Weierstrass data which is  not
of fold-type, but the \red{corresponding} maximal surface admit
only fold singularities:

\begin{example}[A maximal helicoid]\label{ex:helicoid}
If we set
$g=z$ and $\omega=\imag dz/z^2$,
then $(g,\omega)$ induces a maximal helicoid
by setting $z=r e^{\imag \theta}$ ($r>0$,\,\,
$\theta\in \R$),  and
$$
f:=\left(2 \theta,-\left(r+r^{-1}\right) 
\sin \theta,\left(r+r^{-1}\right) \cos \theta
\right)
$$
is defined on the universal covering of
$\C\setminus \{0\}$, and admits only fold singularities at 
the unit circle $r=1$.
However, $(g,\omega)$ is not of fold-type,
since $\{0,\infty\}$  are ends of the Weierstrass data
$(g,\omega)$.
It should be noted that if we set $u=(r+r^{-1})/2(>1)$,
the surface $f$ analytically extends even when $|u|<1$.
Moreover, the entire analytic extension of a 
maximal helicoid coincides exactly with the 
helicoid as a minimal surface in $\R^3$,
and is embedded.
\end{example}

\begin{proposition}
\label{prop:fold}
Let $(g,\omega)$ be a Weierstrass data of
fold-type defined on $S^2$, then the maxface
$f$ induced by $(g,\omega)$
satisfies the period condition
in $\R^3_1$, and gives a weakly
complete maxface having fold singularities 
in $S^1$.
\end{proposition}

\begin{proof}
Let $p_1,...,p_N\in S^1$ be the ends of $(g,\omega)$.
It is sufficient to show that the maxface $f$ induced by $(g,\omega)$
is single-valued on a punctured disk at each $p_j$
($j=1,...,N$).
The Gauss map $g$ does not 
have a branch point at $p_j$.
In fact, if $dg(p_j)=0$ for some
$j=1,...,N$, then the singular set $|g|=1$
must bifurcate at $p_j$, but
this contradicts that the singular set
of $ds^2$ 
(cf. \eqref{eq:dsL})
is exactly equal to $S^1$.
Thus, we can take a local complex coordinate
$(U,\xi)$ centered at $p_j$ such that
$
g=e^{\imag\xi}.
$
In this coordinate, the singular
set of $ds^2$ coincides
with the real axis in the $\xi$-plane.
We set
$
\omega=\hat \omega(\xi) d\xi.
$
Then the condition (3) in 
Definition \ref{def:f-data}
implies the function
$$
\frac{-\imag dg}{g^2\omega}
=\frac{e^{\imag\xi}}{e^{2\imag\xi}\hat \omega}=
\frac{1}{e^{\imag\xi}\hat \omega}
$$
takes real values on the real axis in the
$\xi$-plane.
Thus, we can write  
$
\hat \omega(\xi)=e^{-\imag\xi}h(\xi),
$
where $h(\xi)$ is a
meromorphic function satisfying 
$
\overline{h(\bar \xi)}=h(\xi).
$ Then
$$
(-2g, 1+g^2,\imag (1-g^2))\omega
=2(-1,\cos\xi,\sin \xi)h(\xi)d\xi
$$
has real residue, since $\cos\xi,\sin \xi$,
and $h(\xi)$ are
real-valued functions on the
real axis. 
Thus, $f$ is single-valued
on a neighborhood of $p_j$.
Since $p_j$ is arbitrarily chosen,
$f$ is single-valued
on
$S^2\setminus \{p_1,...,p_N\}$,
proving the assertion,
since the weak completeness of $f$
follows from Proposition \ref{prop:wq}.
\end{proof}

The following assertion holds:

\begin{corollary}
\label{lem:Q}
Let $p$ $($resp.~$q)$ be a fold singular point 
$($resp.~an end$)$ of the weakly complete maxface 
$f$ associated
to a Weierstrass data $(g,\omega)$ 
of fold-type on $S^2$.
Then the Hopf differential $Q:=\omega dg$ does 
not vanish at $p$ $($resp. has a pole at $q)$.
In particular, the umbilic points of $f$ 
correspond to
the zeros of $Q$.
\end{corollary}

\begin{proof}
Since $f$ is weakly complete,  $\omega(p)\ne 0$ holds.
On the other hand,  since $p$ is a non-degenerate
singular point, we have $dg(p)\ne 0$.
Thus we get the assertion for $p$.
We next consider the case that $q$
is an end of $f$. Then, we have $|g(q)|=1$.
If $dg(q)=0$ holds, then the set $|g|=1$
bifurcates at the point $q$, which contradicts
the condition (2) in Definition \ref{def:f-data}.
So we get $dg(q)\ne 0$.
Since $Q=\omega dg$, we have
$$
ds^2_E=(1+|g|^2)^2\frac{|Q|^2}{|dg|^2}.
$$
Since $ds^2_E$ is complete at $z=q$,
the facts $|g(q)|=1$ and $dg(q)\ne 0$
yield that $Q$ must have a pole at $z=q$.
\end{proof}

\section{A characterization of Kobayashi surfaces}
\label{sec2}

We shall prove the following assertion:

\begin{theorem}\label{thm:main}
Let $f:S^2\setminus\{p_1,\dots,p_N\}\to \R^3_1$
be a weakly complete
maxface satisfying the following four conditions: 
\begin{enumerate}
\item \label{i:0}
The Gauss map $g:S^2\setminus\{p_1,\dots,p_N\}\to S^2$
of $f$ has at most 
poles\footnote{
This assumption is actually needed.
If we set $g=e^{\imag z}$ and
$\omega=-e^{-\imag z}dz$, then
$z=\infty$ is an essential singularity 
of $g$, and
the induced maxface 
satisfies (2)--(4) by setting $I(z)=\bar z$.
It is congruent to the helicoid
as in Example \ref{ex:helicoid}
by the coordinate change $\zeta=e^{\imag z}$. 
} 
at $p_1,\dots,p_N$.
\item \label{i:1} There exists an anti-holomorphic involution
$I:S^2\to S^2$ such that
$f\circ I=f$.
\item  \label{i:3}
The ends $p_1,...,p_N$ of $f$ lie on the fixed-point set of $I$.
\item \label{i:2}
The fixed point set $\Sigma^1$ of $I$ 
coincides with the fold singularities of $f$. Moreover, $f$ has
no singularities on $S^2\setminus\Sigma^1$.
\end{enumerate}
Then there exist an integer $n(\ge 2)$ and 
$2n$ real numbers $($called the {\it angular data} of $f)$
\begin{equation}\label{eq:angle}
0=\alpha_0\le \alpha_1\le \cdots \le \alpha_{2n-1}<2\pi
\end{equation}
and complex numbers
$b_1,...,b_{n-1}$
 $($not necessarily mutually distinct$)$
satisfying
\begin{equation}\label{eq:B}
|b_1|,...,|b_{n-1}|<1
\end{equation}
such that $f$ is 
homothetic to
a maxface associated to the
following Weierstrass data of fold-type:
\begin{equation}\label{eq:maxfld}
g=\prod_{i=1}^{n-1} \frac{z-b_i}{1-\overline{b_i}z},\qquad
\omega=
\frac{\imag \prod_{i=1}^{n-1}(1-\overline{b_i}z)^2}
{\prod_{j=0}^{2n-1}(e^{-\imag \alpha_j/2}z-{e^{\imag \alpha_j/2}})}dz.
\end{equation}
The number of distinct values in 
$\{\alpha_0, \alpha_1, \cdots , \alpha_{2n-1}\}$
coincides with $N$.
More precisely,  the set 
$\{e^{\imag \alpha_0},...,e^{\imag \alpha_{2n-1}}\}$
coincides with the set $\{p_1,...,p_{N}\}$.
Conversely, for each angular data satisfying \eqref{eq:angle},
we get a maxface satisfying $(1)-(4)$ by 
choosing the Weierstrass data
\eqref{eq:maxfld}.
\end{theorem}

\begin{proof}
Without loss of generality, we may set $I$ 
as the canonical inversion
$$
I(z)=1/\bar z.
$$
In this case the fixed point set $\Sigma^1$  of $I$
is the unit circle $S^1:=\{z\in \C\,;\, |z|=1\}$.
By \ref{i:1}, we have $f(1/\bar z)=f(z)$, and so
$$
\op{Re}(\overline{F(1/\bar z)})=\op{Re}(F(z))
$$
holds, where $F$ is the holomorphic lift of $f$.
By the identity theorem, we can conclude that
$$
\overline{F(1/\bar z)}=F(z)+\imag C 
$$
holds for some constant vector $C\in\R^3$.
In particular,
$\overline{d(F\circ I)}=dF$ holds.
Let $(g,\omega)$ be the Weierstrass data
of $f$. By \ref{i:0}, $g$ is a meromorphic
function on $S^2$.
Since $f$ is weakly complete,
the metric $ds^2_E$ 
given in \eqref{eq:dse}
is complete.
The well-known completeness lemma
(see the introduction of \cite{UY2}) 
yields that $\omega$
is a meromorphic $1$-form on $S^2$,
and so $(g,\omega)$ is a Weierstrass
data on $S^2$.
Since $f$ admits only fold singularities 
which lie in $S^1$
by \ref{i:2} and \ref{i:3},
$(g,\omega)$ satisfies (ii) and (iii) of 
Definition~\ref{def:f-data}.
On the other hand, the condition \ref{i:3} 
corresponds to (i) of Definition~\ref{def:f-data}.
Thus, we can conclude that $(g,\omega)$ is of fold-type.
Using \eqref{eq:W0}, we have
\begin{equation}\label{eq:g1g}
\overline{g\circ I}
=-\frac{\overline{d(F_0\circ I)}}
{\overline{d(F_1 \circ I)} +\imag \overline{d(F_2\circ I)}}
=-\frac{dF_0}{dF_1+\imag dF_2}
=-\frac{dF_0 (dF_1-\imag d F_2)}{(dF_0)^2}=\frac1g,
\end{equation}
where $F=(F_0,F_1,F_2)$.
We let $b_1,...,b_{l}$ be all the zeros of $g$ as a 
meromorphic function on
$S^2$.  
We do not need to assume that these 
numbers are distinct.
For example if $b_1=b_2$, 
then $g$ has zeros of order more than one. 
By \eqref{eq:g1g},
$$
1/\overline{b_1},\dots,1/\overline{b_l}
$$
must be the set of all poles of $g$, 
where $1/0$ would be regarded as
$\infty$.
By the fact $g\circ I=1/\overline{g}$, 
we may set $|b_1|<1$,
without loss of generality.
So we may set
\begin{equation}\label{eq:cg}
g=c \prod_{i=1}^l\frac{z-b_i}{1-\bar b_i z}
\qquad (|c|=1).
\end{equation}

On the other hand, since $Q$ is the $(2,0)$-component of
the complexification of the second fundamental form, 
$Q$ has no poles at each regular point of $f$.
By this together with Corollary \ref{lem:Q}, 
we can conclude that
$Q$ has a pole at $z=q$ 
if and only if $q$ is an end of $f$.
Let $\tilde n$ be the total sum of order of poles
of $Q:=\omega dg$,
and $m$ the total order of zeros of $Q$
which are contained in the interior of the unit disk.
Since $Q$ has no zeros on $S^1$
by Corollary \ref{lem:Q},
the fact that $Q\circ I=Q$ yields that
the total sum of the orders of 
zeros of $Q$ is equal to
$2m$. By the Riemann-Roch relation for $Q$
we have
\begin{equation}\label{eq:RR}
2m-\tilde n=-2\chi(S^2)=-4.
\end{equation}
In particular, $\tilde n$ is an even number,
and can be denoted by $\tilde n=2n$.
Also, as shown in the proof of Corollary \ref{lem:Q},
$dg(q)\ne 0$ holds at each end $q$.
So the set of poles of $Q$ coincides with the
set of poles of $\omega$, and the total sum of
orders of poles of $\omega$ is equal to $\tilde n(=2n)$.
Moreover, since all ends of $f$ lie in $S^1$,
%we may assume that
there exist real numbers 
$\alpha_1,...,\alpha_{2n-1}$
satisfying \eqref{eq:angle} such that
$
e^{\imag \alpha_j}
$ 
($j=0,1,...,2n-1$)
are poles of $\omega$
counted with multiplicity.
In particular,
$\{e^{\imag \alpha_0},...,e^{\imag \alpha_{2n-1}}\}$
coincides with $\{p_1,...,p_N\}$.
\red{We now use the parameter change
\begin{equation}\label{eq:g000}
z\mapsto e^{-\imag \alpha_0}z.
\end{equation}
to conclude that we may assume $\alpha_0=0$.
We next consider the following change of the
Weierstrass data given by
\begin{equation}\label{eq:g-omega}
(g,\omega)\mapsto (a g,\bar a \omega)\qquad
(a\in S^1),
\end{equation}
which preserves
the first fundamental form $ds^2$ as in 
\eqref{eq:dsL} as well as the second fundamental form
$Q+\overline{Q}$, where $Q:=\omega dg$, and
gives the same maximal
surface up to a rigid motion of $\R^3_1$.
So we may assume that $c=1$ in the expression \eqref{eq:cg}.}

On the other hand, if $\omega$ has a zero at a regular 
point of $f$, it must be a pole of $g$. 
Since $1/\overline{b_i}$ ($i=1,...,l$) 
are regular points of $f$,
the regularity of $ds^2$ at $1/\overline{b_i}$ 
implies that
$\omega$  has the factor $(z-1/\overline{b_i})^2$
for each $i=1,...,l$.
Thus, we may set
\begin{align*}
\omega
&
=\frac{B\prod_{i=1}^l (1-\overline{b_i}z)^2}
{\prod_{j=0}^{2n-1}\left(e^{-\imag \alpha_j/2}z
-e^{\imag \alpha_j/2}\right)}dz 
\\
&=B \frac{\prod_{i=1}^l (1-\overline{b_i}z)^2}
{z^n\prod_{j=0}^{2n-1}\left(e^{-\imag \alpha_j/2}z^{1/2}
-e^{\imag \alpha_j/2}z^{-1/2}\right)}dz,
\end{align*}
where $B\in \C\setminus \{0\}$.
For the sake of simplicity,
we set 
\begin{equation}\label{eq:gamma}
c_j:=e^{\imag \alpha_j/2}
\qquad
(j=0,...,2n-1).
\end{equation}
Then
\begin{align*}
\frac{dg}{g^2\omega}
&=\frac1{g\omega}\frac{dg}{g}
=
\frac{z^{n}}B
\frac{\prod^l_{i=1}(1-\bar b_iz)}{\prod^l_{i=1}(z-b_i)}
\frac{\prod^{2n-1}_{j=0}(\bar c_j \sqrt{z}-c_j/\sqrt{z})}
{\prod^l_{i=1}(1-\bar b_iz)^2}
\frac{dg/dz}{g}\\
&=
\frac{z^{n-l}}{B}
\frac{1}{\prod^l_{i=1}(z-b_i)}
\frac{\prod^{2n-1}_{j=0}(\bar c_j \sqrt{z}-c_j/\sqrt{z})}
{\prod^l_{i=1}(z^{-1}-\bar b_i)}
\frac{dg/dz}{g}.
\end{align*}
If we set $z=e^{\imag \theta}$ ($\theta\in \R$),
then
$$
\prod^l_{i=1}(z-b_i)(z^{-1}-\bar b_i),
\qquad
\prod^{2n-1}_{j=0}(\bar c_j \sqrt{z}-c_j/\sqrt{z})
$$
are real-valued. 
So there exists a function $C_1(z)$ 
satisfying $C_1(e^{\imag t})\in \R$
for $t\in \R$
such that
$$
\frac{dg}{g^2\omega}
=C_1
\frac{z^{n-l}}{B}
\frac{dg/dz}{g}.
$$
Since
$$
\frac{dg/dz}{g}=\frac{d(\log g)}{dz}=\sum_{i=1}^l
\frac{1-|b_i|^2}{(z-b_i)(1-\overline{b_{i}}z)}
=\frac{1}{z}\sum_{i=1}^l
\frac{1-|b_i|^2}{(z-b_i)(z^{-1}-\overline{b_i})},
$$
the function
$C_2(z):=z {g_z}/{g}$
has the property that
$C_2(e^{\imag t})\in \R$
for $t\in \R$, where $g_z=dg/dz$.
Hence, we can write
$$
\frac{dg}{g^2\omega}
=C_1C_2 
\frac{z^{n-l-1}}{B}.
$$
On the other hand,
${dg}/{g^2\omega}$
is  $\imag \R$-valued on $S^1$
because of  (iii) of Definition \ref{def:f-data}.
Thus, we can conclude that
$
z^{n-l-1}/{B}
$
must also be $\imag \R$-valued for all $z\in S^1$.
In particular, we have
$
n-l-1=0
$
and $B\in \imag \R$.
By a homothety of $f$, we may assume that
$
B=\imag.
$
Since $0$ is neither a pole nor a zero
of $\omega$,
we have  $l=n-1$,
and so we get the expression
\eqref{eq:maxfld}.
Suppose that $b_1,...,b_{n-1}$ satisfies 
\eqref{eq:B},
then the fact
$$
\left|\frac{z-b_i}{1-\bar b_i z} \right|<1
\qquad(|z|<1)
$$
yields that $|g(z)|<1$ whenever $|z|<1$.
By our definition,
$|b_1|<1$ and
$g(z)=0$ if
$z=b_1$. 
On the other hand, if one of $b_i$ 
($i=2,...,n-1$)
satisfies
$|b_i|>1$, then
$g$ has a pole at $z=1/\bar b_i$.
In particular, $|g(z)|>1$ if $z$ is sufficiently
close to $1/\bar b_i$.
Since $b_1$ and
$1/\bar b_i$ lie in the unit disk,
by the intermediate value 
theorem, there exists a point $z_0$ ($|z_0|<1$)
such that $|g(z_0)|=1$.
Hence, all of $b_1,...,b_{n-1}$ must lie in the unit disk.
Conversely, one can easily verify that
the Weierstrass data $(g,\omega)$
given in \eqref{eq:maxfld} is of fold-type,
and induces a maxface
satisfying (1)--(4)
by Proposition~\ref{prop:fold}.
\end{proof}

\begin{definition}\label{def:K}
We call
a maxface given in this theorem a 
{\it Kobayashi surface} of order $n$. 
We denote by $\mathcal K_n$ the set of 
congruent classes of all Kobayashi surfaces
of order $n$.
\end{definition}

\begin{remark}\label{eq:rotation}
\red{As seen in the proof of Theorem \ref{thm:main}
(cf. \eqref{eq:g000} and \eqref{eq:g-omega}),
the circular rotation of the angular data 
$
(\alpha_0,....,\alpha_{2n-1})\mapsto
(\alpha_1,....,\alpha_{2n-1},\alpha_0)
$
and the translation
$
(\alpha_0,....,\alpha_{2n-1})\mapsto
(\alpha_0+\beta,....,\alpha_{2n-1}+\beta)
$
($\beta\in \R$)
do not change the resulting Kobayashi surface
whenever $b_1=\cdots=b_n=0$.
}
\end{remark}

\begin{corollary}\label{cor:main}
Let $f$ be a Kobayashi surface 
as in Theorem \ref{thm:main}.
If $f$ has at most two umbilics,
then there exist an integer $n(\ge 2)$ and 
$2n$ real numbers $(\alpha_0,\dots,\alpha_{2n-1})$
satisfying \eqref{eq:angle}
such that $f$ is
homothetic to a maxface associated to the
following Weierstrass data of fold-type:
\begin{equation}\label{eq:maxfld2}
g=z^{n-1},\qquad
\omega=
\frac{\imag \Lambda dz}{\prod_{j=0}^{2n-1}
(z-e^{\imag \alpha_j})},
\end{equation}
where
\begin{equation}\label{eq:beta}
\Lambda:=\exp\imag 
\left(\frac{\alpha_0+\cdots+\alpha_{2n-1}}2\right).
\end{equation}
\end{corollary}

\begin{proof}
Let $n$ be the order of the Kobayashi surface $f$.
If $n=2$, then $f$ has no umbilics (cf.~Corollary \ref{lem:Q}) and the
degree of Gauss map is equal to $1$.
In this case,
by a suitable M\"obius transformation in $S^2$
preserving the unit circle $S^1$,
we may assume that $g=z$.
Namely, we may set $b_1=0$, and get the assertion.
So we may assume that $n\ge 3$,
Since $Q$ has no zeros
on $S^1$ (cf. Corollary \ref{lem:Q}),
the fact that $f\circ I=f$ implies  
there are at least two umbilics on $f$. 
So by our assumption, 
the number of umbilics must be exactly two.
By a suitable motion in $\R^3_1$,
we may assume that $g(0)=0$, that is,
$b_1=0$.
By a suitable M\"obius transformation in $S^2$
preserving the unit circle $S^1$,
we may also assume that $z=0$
is an umbilic of $f$.
In this case,  we may set
\begin{equation}\label{eq:star}
b_{1}=\cdots=b_{\mu}=0,\qquad
b_{\mu+1},...,b_{n-1}\ne 0,
\end{equation}
where $2\le \mu\le n-1$. 
If  $\mu=n-1$, then $b_i=0$ ($i=1,...,n-1$)
and $g=z^{n-1}$. Suppose that $\mu<n-1$, by way
of contradiction.
Using \eqref{eq:maxfld}  we write
$g=z^\mu X/Y$,
where
$$
X:=\prod_{i=\mu+1}^{n-1} (z-b_i),\qquad
Y:=\prod_{i=\mu+1}^{n-1} (1-\overline{b_i}z).
$$
Since
\begin{equation}\label{eq:g1}
dg=\frac{\mu z^{\mu-1}XYdz+z^{\mu}(dX)Y
-z^{\mu}X(dY)}{Y^2},\quad 
\omega=\frac{\imag Y^2dz}
{\prod_{j=0}^{2n-1} (e^{-\imag \alpha_j/2}z
-{e^{\imag \alpha_j/2}})},
\end{equation}
we have
$$
Q:=\omega dg=
\imag\frac{\mu z^{\mu-1}XYdz+z^{\mu}(dX)Y
-z^{\mu}X(dY)}
{\prod_{j=0}^{2n-1} (e^{-\imag \alpha_j/2}z
-{e^{\imag \alpha_j/2}})}dz.
$$
Since the zeros of $Y$ are not umbilics 
of $f$, the umbilics of $f$
correspond to the zeros of $dg$  (cf. Corollary \ref{lem:Q}).
Since $f$ has exactly two umbilics, 
we can write  
\begin{equation}\label{eq:g2}
dg=\frac{C(z)z^{n-2}}{Y^2}dz\qquad (C(z)\ne 0).
\end{equation}
Comparing \eqref{eq:g1} and \eqref{eq:g2},
we have that
\begin{equation}\label{eq:identity}
C(z)z^{n-\mu-1}=\mu XY+z (dX)Y-zX(dY).
\end{equation}
By
\eqref{eq:star},
the right hand side does not vanish at $z=0$,
a contradiction. Hence, we have
$\mu=n-1$.
In particular, \eqref{eq:maxfld} reduces to
\eqref{eq:maxfld2}.
\end{proof}

A Kobayashi surface
(cf. Definition \ref{def:K})
 associated to the Weierstrass data
as in \eqref{eq:maxfld2} is called 
a {\it Kobayashi surface of principal type}
with angular data $(\alpha_0,...,\alpha_{2n-1})$.
In other words,
Kobayashi surfaces of principal type
are obtained by setting
$b_1=\cdots=b_{n-1}=0$
in \eqref{eq:maxfld}.
We denote by $\mathcal K_n^0(\subset \mathcal K_n)$ 
the set of Kobayashi surfaces 
of principal type (cf. Definition \ref{def:K} for
$\mathcal K_n$). 
A Kobayashi surface which is not of principal type
is called of {\it general type}.

\begin{proposition}\label{prop:4}
The
set $\mathcal K_2$ of Kobayashi surfaces
of order $2$ 
has $3$ degrees of freedom,
and the set  
$\mathcal K_n$
Kobayashi surfaces
of order $n(\ge 3)$
has $4n-7$ degrees of freedom,
up to congruence and homothety.
Moreover, 
there are $2n-1$ degrees of freedom
for $\mathcal K_n^0$,
up to congruence and homothety.
Furthermore,  
Kobayashi surfaces
of order less than $4$
are all of principal type.
\end{proposition}

\begin{proof}
We suppose $n\ge 3$.
Then $f$ has umbilics.
Let $z=b$ be such an umbilic of $f$, 
then it is a zero of $Q$.
We may assume that 
the maximum value of the orders of zeros of $Q$
is attained at $z=b$.
By a suitable M\"obius transformation,
we may set $b=0$.
Then it holds that $dg(0)=0$.
Without loss of generality, we may also assume 
$g(0)=0$,
and have the following expression
like as in the proof of Corollary~\ref{cor:main};
\begin{equation}\label{eq:maxfld0}
g=z^{\mu}\prod_{i=1}^{n-\mu-1} \frac{z-b_i}{1-\overline{b_i} z},\,\,\,
\omega=
\frac{\imag \prod_{i=1}^{n-\mu-1}(1-\overline{b_i} z)^2}
{\prod_{j=0}^{2n-1}
(e^{-\imag \alpha_j/2}z-e^{\imag \alpha_j/2})}dz
\quad (2\le \mu\le n-1).
\end{equation}
We may fix $\mu$ equal to $2$ in \eqref{eq:maxfld0}
by allowing some of the $b_i$ to be zero. Thus 
\eqref{eq:maxfld0} has $n-3$ complex parameters 
$b_1,...,b_{n-3}$ in the unit disk and $2n-1$ 
(real) parameters
$\alpha_1,....,\alpha_{2n-1}$ on the interval $[0,2\pi)$.
Therefore,  
$\mathcal K_n$ ($n\ge 3$)
has
$4n-7$ degrees of freedom,
up to congruence and homothety.
On the other hand, the 
surfaces of principal type satisfy $\mu=n-1$,
and  they can be controlled by
the $2n-1$ angular parameters
$\alpha_1,...,\alpha_{2n-1}$, since $\alpha_0=0$.
If $n=3$, 
$$
2\le \mu \le n-1=2
$$
holds, and so $\mu=n-1$, in particular, $\mathcal K_3$
consists only of surfaces of principal type.

If $n=2$,  we can normalize $b_1=0$
by a suitable M\"obius transformation fixing
the unit circle.
Then, we may set 
\begin{equation}\label{neq2}
g=z,\quad
\omega=\frac{\imag e^{\imag (\alpha_1+\alpha_2+\alpha_3)/2}}
{(z-1)(z-e^{\imag \alpha_1})(z-e^{\imag \alpha_2})
(z-e^{\imag \alpha_3})} dz,
\end{equation}
that is, $\mathcal K_2$
consists only of surfaces of principal type,
and  can be controlled  by the $3$ parameters
$\alpha_1,\alpha_{2},\alpha_{3}$.
\end{proof}

\begin{example}[Scherk-type surfaces]\label{ex:Sgen}
We set
\begin{equation}\label{angle:S}
\alpha_j:=\frac{\pi j}{n} \qquad (j=0,...,2n-1).
\end{equation}
Regarding the fact that
$\sum_{j=0}^{2n-1}\alpha_j=(2n-1)\pi$,
\eqref{eq:maxfld2}
reduces to
\begin{equation}\label{eq:WKar}
g=z^{n-1},\qquad
\omega=\frac{dz}{z^{2n}-1}
\qquad (n=2,3,4,...).
\end{equation}
Its companion  $(-\imag g,\omega)$
is the Weierstrass data of the minimal surface
called the {\it Scherk
saddle tower} given in \cite{Kar}. 
The maximal surface $\mathcal S_n$ induced by 
\eqref{eq:WKar}
is the most important example of a
Kobayashi surface of principal type.
(The fact that $\mathcal S_2$ induces an entire graph
was shown in Kobayashi \cite{K}, see also
Example \ref{ex:Sch}.)
We shall show later that $\mathcal S_n$ can be  real analytically
extended as an entire graph for $n\ge 3$. 
\end{example}

\begin{example}[Jorge-Meeks type  maximal surfaces]
\label{ex:J}
We set
\begin{equation}\label{angle:J}
\alpha_{2j}=\alpha_{2j+1}:=\frac{2\pi j}{n} \qquad 
(j=0,1,...,n-1).
\end{equation}
Then 
\eqref{eq:maxfld2}
reduces to
\begin{equation}\label{eq:WJM}
g=z^{n-1},\qquad
\omega=\frac{\imag dz}{(z^n-1)^2}
\qquad (n=2,3,4,...),
\end{equation}
and its companion
 $(-\imag g,\omega)$
is the Weierstrass data of the minimal surface
called the Jorge-Meeks $n$-noid. 
So,  the induced maxface 
is called the
{\it Jorge-Meeks type maximal surface},
and we denote it by $\mathcal J_n$.
As shown in the authors' previous work \cite{FKKRUY},
the analytic extension of $\mathcal J_n$ is properly embedded
for all $n\ge 2$.
Although $\mathcal J_n$ ($n\ge 3$) is not,
the surface $\mathcal J_2$ is an entire graph
of mixed type, given in \red{\eqref{eq:J0}}.
\end{example}

We next consider the Kobayashi surface satisfying
$
\{\alpha_0,\alpha_1,\alpha_2,\alpha_3\}\subset \{0,\pi\}
$.
%Such a surface is symmetric with respect to
%a time-like plane.
The surface of type $(0,0,\pi,\pi)$ is $\mathcal J_2$
as in Example \ref{ex:J}.
The most degenerate case is of type $(0,0,0,0)$
as follows:
 
\begin{example}[The ruled Enneper surface]
\label{exa:0000}
In case of 
$(\alpha_0,\alpha_1,\alpha_2,\alpha_3)=(0,0,0,0)$,
we have
$$
g=z,\quad \omega=\frac{\imag dz}{(1-z)^4},
$$
and it induces the following Kobayashi surface
\begin{align*}
f&=
\biggl(
\frac{\sin \theta (-6 u \cos \theta+\cos (2 \theta)+5)}
{12 (\cos \theta-u)^3},
\frac{\sin \theta \left(-3 u \cos \theta+\cos (2 \theta)
+3 u^2-1\right)}{6 (\cos \theta-u)^3},\\
&\phantom{aaaaaaaaaaaaaaaaaaaaaaaaaaaaaaaaaaaaaaaaaaaaaaa}
-\frac{u \cos \theta-1}{2 (u-\cos \theta)^2}\biggr),
\end{align*}
where
$z=r e^{\imag \theta}$ and $u:=(r+r^{-1})/2$.
Then $f$ can be extended to the region 
$\{(u,\theta)\,;\, 1>u>\cos\theta\}$,
which is denoted by $\tilde f$.
The image of $\tilde f$ coincides with the 
set satisfying
(see Figure \ref{fig:00}, left)
$$
\Phi=0,\quad
\Phi:=
\frac{4 t^3}{3}+4 t^2 x+4 t x^2-2 t y+t+\frac{4 x^3}{3}-2 x y.
$$
Since
$$
\Phi_y=-2(x+t), \quad \Phi_t|_{(t,x,y)=(t,-t,y)}=1-2y,\quad
\Phi_x|_{(t,x,y)=(t,-t,y)}=-2y,
$$
the gradient vector $(\Phi_t,\Phi_x,\Phi_y)$
never vanishes on $\R^3$.
Hence the set $\Phi=0$ has no self-intersections.
Thus the analytic extension of the Kobayashi
surface with
angular parameter $(0,0,0,0)$ is
embedded, that is, the problem
is affirmative even for this case. 
The resulting surface is a ruled surface
found in Kobayashi \cite{K},
and is called the conjugate of 
Enneper's surface of 2nd kind.
In fact, $\Phi(t,x,y)=0$
implies that 
$$
\Phi\left(t-k,x+k,y-\frac{k}{2(t+x)}\right)=0
$$
for all $k\in \R$.
Space-like ruled maximal surfaces
in $\R^3_1$ are classified in \cite{K}.
\end{example}

\begin{figure}[htb]
 \centering
 \includegraphics[height=4.8cm]{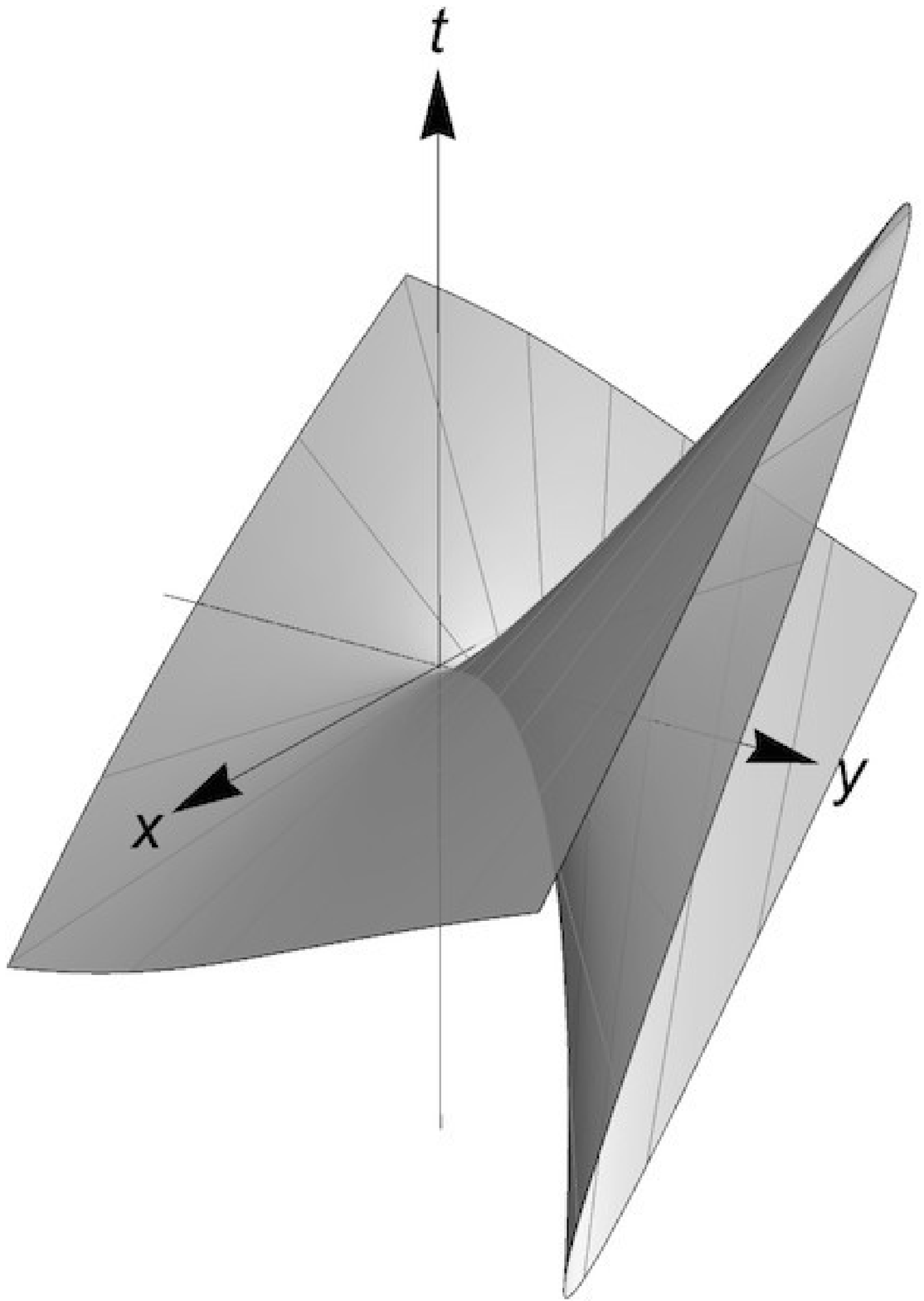}  
\qquad
 \includegraphics[height=4.8cm]{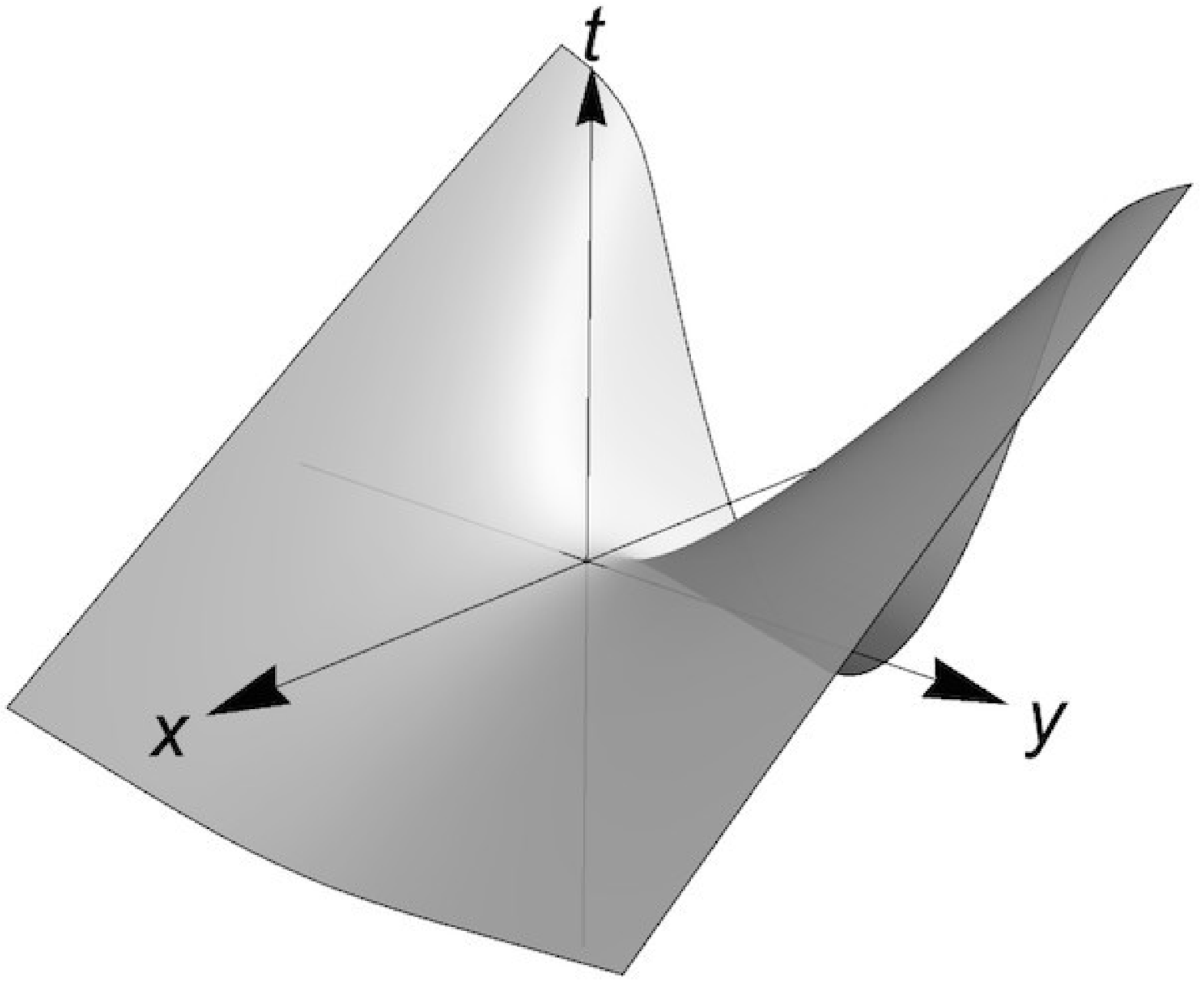}  
\caption{%
Example \ref{exa:0000} (left)
and
Example \ref{angle:third} (right)
}
\label{fig:00}
\end{figure}

\begin{example}
\label{angle:third} 
We set $n=2$ and 
$\alpha_0=\alpha_1=\alpha_2=0$
and $\alpha_3=\pi$.
Then
\eqref{eq:maxfld2}
for $b_1=0$
reduces to
\begin{equation}\label{eq:3}
g=z,\qquad
\omega=\frac{-dz}{(z-1)^3(z+1)}.
\end{equation}
By setting $z=r e^{\imag \theta}$,
the induced maximal surface is given by
\begin{equation}\label{eq:f0000}
f=
\frac18\left(
A+B
,A-B, -\frac{8r \sin \theta}{D_-}\right),
\end{equation}
where
$$
A=-4 \frac{\left(r^3+r\right) \cos \theta}{D_-^2}
+\frac{8 r^2}{D_-^2},\quad
B=\log \frac{D_+}{D_-},\quad
D_\pm=r^2\pm 2 r \cos \theta+1.
$$
In particular,  the image of $f$
lies in the set (see Figure \ref{fig:00})
$$
\mathcal G=\left\{(t,x,y)\in \R^3_1\,;\,
\Phi=0\right\},
\quad
\Phi:=
\frac{1}{2} (e^{4 (t+x)}-1)+2 (t-x)-4 y^2=0.
$$
Since $\Phi_t=2 + 2 e^{4 (t+x)}\ge 2(>0)$,
\cite[Corollary 1]{ZG} yields that
$\mathcal G$ can be realized as the image of
an entire graph of $x,y$.
This implies that $\mathcal G$ can be considered
as the analytic extension of $f$.
It should be remarked that Akamine \cite{A} 
independently found this surface.
Moreover, he showed that the surface
can be foliated by parabolas. 
\end{example}

\section{Analytic extension of Kobayashi surfaces}

A Kobayashi surface
 $f$ of order $n$ is 
invariant under the symmetry
$r \mapsto 1/r$, where $z=r e^{\imag \theta}$.
The singular set $S^1:=\{|z|=1\}$ of $f$
coincides with the fixed point set under the symmetry.
So, like as the case of the Jorge-Meeks type maximal
surface (cf. \eqref{eq:WJM})
discussed in the authors' previous work
\cite{FKKRUY},
it is natural to expect that we can introduce a 
new variable $u$ by
 \begin{equation}\label{eq:u-r}
    u := \frac{r+r^{-1}}{2},
\end{equation}
which is invariant under the symmetry
$r \mapsto 1/r$.
We set
$$
\bar D_1^*:=\{z\in \C\,;\, 0<|z|\le 1\}.
$$
To obtain the analytic extension of the image of $f$,
we define an analytic map
$$
\iota:
\bar D_1^*\ni z=r e^{\imag \theta} \mapsto 
(\frac{r+r^{-1}}2,\theta)\in 
\R\times \R/2\pi \Z.
$$
The image of the map $\iota$ is
given by
$$
\hat \Omega_n:=\{(u,\theta)\in 
\R\times \R/2\pi \Z\,;\,
u \ge 1
\}.
$$
The map $\iota$ is bijective,
whose inverse is given by 
$$
\iota^{-1}:
\hat \Omega_n\ni (u,\theta)
\mapsto 
(u-\sqrt{u^2-1},\theta)\in \bar D_1^*.
$$
If we regard $f$ as a function of $(u,\theta)$,
the origin $z=0$ in the original source space of 
$f$ does not lie in the $(u,\theta)$-plane.
To indicate what the origin in the old
complex coordinate $z$ becomes in the
new real coordinates $(u,\theta)$,
we attach a new point $p_{\infty}$
to $\hat \Omega_n$
as the `point at infinity',
and extend the map $\iota$ so that
$$
\iota(0)=p_\infty.
$$
Hence we have a one-to-one correspondence 
between $\{|z| \le 1 \}$ 
and $\hat \Omega_n\cup\{p_\infty\}$.
In particular, 
$\hat \Omega_n\cup\{p_\infty\}$ can be 
considered as  an analytic $2$-manifold.
The purpose of this section is
to prove the following:

\begin{theorem}\label{thm:domain}
Let $f$ be a Kobayashi surface
associated to the Weierstrass data as in
\eqref{eq:maxfld}.
Then  
$
df(z)=\Re(-2g,1+g^2,\imag(1-g^2))\omega
$ 
can be parametrized using new parameters
$u=(r+r^{-1})/2$ and $\theta$,
where $z=r e^{\imag \theta}$
$(r>0,\theta\in [0,2\pi))$.
Moreover, $df \circ \iota$ can be analytically extended
on the set
\begin{equation}\label{eq:omega2}
\Omega_{\alpha_{0},...,\alpha_{2n-1}} :=
   \left\{
    (u,\theta)\in \R\times \R/2\pi\Z\,;\,
      u> \max_{j=1,\dots,N}
             \left[\cos\left(\theta-\alpha_{i_j}\right)\right]
     \right\}\cup \{p_\infty\}, 
\end{equation}%
where
\begin{equation}\label{red-angle}
0=\alpha_{i_1}<\cdots <\alpha_{i_N}<2\pi
\qquad (0=i_1<i_2<\cdots<i_N\le 2n-1)
\end{equation}
are the distinct angular values in \eqref{eq:angle}.
\end{theorem}

For the sake of simplicity, we set
\begin{equation}\label{eq:beta-2n}
\beta_{j}:=\alpha_{i_{j+1}} \qquad (j=0,...,N-1),\quad
\beta_{N}:=2\pi
\end{equation}
and (cf.~Figure~\ref{fig:region})
\begin{equation}\label{eq:beta-interval}
\begin{aligned}
   \gamma_j&:=
       \frac{\beta_j+\beta_{j+1}}{2},\qquad (j=0,\dots,N-1),\\
   I_j&:= \begin{cases}
	  \bigl[
	    \gamma_{j-1},\gamma_{j}
	   \bigr]\qquad &(j=1,\dots,N-1),\\
	   [0,\gamma_0] \cup [\gamma_{N-1},2\pi]
	  \qquad &(j=0).
	 \end{cases}
\end{aligned}
\end{equation}
\begin{figure}
 \centering
 \includegraphics[width=0.7\textwidth]{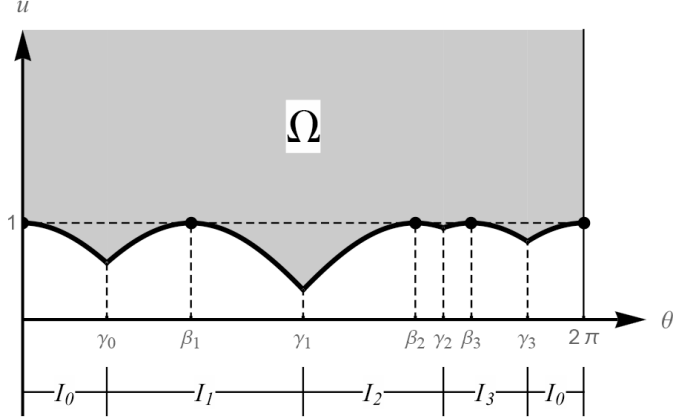}
 \caption{
The figure of
$\Omega:=\Omega_{\alpha_{0},...,\alpha_{3}}$
for $n=2$ and $N=4$} \label{fig:region}
\end{figure}

To prove the theorem, we prepare the
following lemma:

\begin{lemma}\label{lem:interval}
For each pair of numbers  $i,j\in \{0,\dots,N-1\}$ 
$(i\neq j)$, the following inequality holds
\begin{equation}\label{eq3-2}
    \cos(\theta-\beta_i) \ge \cos(\theta-\beta_j)\qquad 
     \qquad (\theta\in I_i).
\end{equation}
Moreover, equality holds if and only if
$(j,\theta)=(i-1,\gamma_{i-1})$ 
or $(j,\theta)=(i+1,\gamma_{i})$.
\end{lemma}

\begin{proof}
Instead of $\R$, we consider $\R/2\pi\Z$, and
then the assertion is invariant
under the cyclic permutation of $\beta_0,...,\beta_{N-1}$.
So, it is sufficient to consider the case that
$i \ge 1$ and $j\ge 2$.
It holds that
 \begin{equation}\label{eq:cos-sub}
  \cos(\theta-\beta_i) - \cos(\theta-\beta_j)
  = 2\sin\left(\frac{\beta_i-\beta_j}{2}\right)
      \sin\left(\theta-\frac{\beta_i+\beta_j}{2}\right).
 \end{equation}
Since $\theta\in I_i$, we have
 \[
    \theta-\frac{\beta_i+\beta_j}{2}
     \in\left[\gamma_{i-1}-\frac{\beta_i+\beta_j}{2},
              \gamma_{i}-\frac{\beta_i+\beta_j}{2}\right]
      =\left[\frac{\beta_{i-1}-\beta_j}{2},
              \frac{\beta_{i+1}-\beta_j}{2}\right].
 \]

If  $j\leq i-1$, then 
$\beta_j\leq \beta_{i-1}<\beta_i$ and $\beta_{i+1}-\beta_j<2\pi$
hold, and
$$
    \theta-\frac{\beta_i+\beta_j}{2}
    \in \left[\frac{\beta_{i-1}-\beta_j}{2},
              \frac{\beta_{i+1}-\beta_j}{2}\right)
    \subset \left[0,\pi\right),\qquad
    \frac{\beta_{i}-\beta_{j}}{2}\in (0,\pi).
$$
Hence, by \eqref{eq:cos-sub} we have
\eqref{eq3-2}.
Equality holds if and only if
$j=i-1$ and $\theta=\gamma_{i-1}$.

On the other hand, if $j\geq i+1$, then 
$\beta_j\geq \beta_{i+1}>\beta_i$ and
$\beta_{i-1}-\beta_j>-2\pi$
hold. Thus we have
$$
    \theta-\frac{\beta_i+\beta_j}{2}
    \in \left(\frac{\beta_{i-1}-\beta_j}{2},
              \frac{\beta_{i+1}-\beta_j}{2}\right]
    \subset \left(-\pi,0\right],\qquad
    \frac{\beta_{i}-\beta_{j}}{2}\in (-\pi,0).
$$
So \eqref{eq:cos-sub} implies 
\eqref{eq3-2} and the equality condition.
\end{proof}

As a consequence, the following two assertions hold:

\begin{corollary}\label{cor:Omega}
Suppose that $\theta\in I_j$.
Then $(u,\theta)\in 
\Omega_{\alpha_{0},...,\alpha_{2n-1}}$
if and only if $u>\cos(\theta-\beta_j)$.
\end{corollary}

\begin{corollary}\label{cor:Omega2}
If $(u,\theta)\in 
\Omega_{\alpha_{0},...,\alpha_{2n-1}}$
then
$$
u>
\min_{j=0,1,...,2n-1}
\left\{\cos\left(\frac{\alpha_{j+1}-\alpha_{j}}2\right)\right\}
\qquad
 (\alpha_{2n}:=2\pi).
$$
\end{corollary}

\begin{proof}
Without loss of generality, we may assume that
$\theta\in I_1$.
Then $\cos(\theta-\beta_1)$ takes a minimum at 
$\theta=\gamma_0$ or $\theta=\gamma_{1}$.
So we have
\begin{align*}
u&>\min\left\{\cos(\gamma_0-\beta_1),\cos(\gamma_{1}-\beta_1)\right\} \\
&=
\min\left\{\cos\left(\frac{\beta_{1}-\beta_0}2\right),
\cos\left(\frac{\beta_{2}-\beta_{1}}2\right)\right\}
\ge
\min_{j=0,1,...,2n-1}
\left\{\cos\left(\frac{\alpha_{j+1}-\alpha_{j}}2\right)\right\}.
\end{align*}
\end{proof}

\begin{proof}[Proof of Theorem \ref{thm:domain}]
We set $2f_zdz=(\phi_0,\phi_1,\phi_2)$.
Then, it holds 
that
\begin{equation}\label{eq:phi}
   \phi_0 := - 2 g \omega,\quad
   \phi_1 := (1+g^2)\omega,\quad
   \phi_2 := \imag (1-g^2)\omega.
\end{equation}
Using \eqref{eq:maxfld},
we have that
\begin{equation}\label{eq:phi0}
\phi_k(=\hat \phi_kdz)
=\frac{p_k(z)}{q(z)}dz\qquad (k=0,1,2),
\end{equation}
where
\begin{align}
\label{eq:Q1}
q(z)&:=
\prod_{j=0}^{2n-1}
(e^{-\imag \alpha_j/2}z-e^{\imag \alpha_j/2}),\\
\label{eq:Q2}
p_0(z)&:=-2\imag \prod_{i=1}^{n-1}
\biggl(
(z-b_i)(1-\bar b_iz)
\biggr), \\
\label{eq:Q3}
p_k(z)&:=\imag^k 
\left(
\prod_{i=1}^{n-1}(1-\bar b_iz)^2
-(-1)^k\prod_{i=1}^{n-1}(z- b_i)^2
\right)\qquad (k=1,2).
\end{align}
Moreover, it holds that
\begin{equation}\label{eq:q}
|q(z)|^2=4^n r^{2n} \prod_{j=0}^{2n-1}
\biggl(u-\cos(\theta-\alpha_j)\biggr)^{2}
\qquad (u=\frac{r+r^{-1}}2),
\end{equation}
where $z=r e^{\imag \theta}$.
We now fix an index $k\in \{0,1,2\}$,
and set
$$
p(z):=p_k(z),\qquad \phi(z):=\phi_k(z).
$$
It holds that
\begin{equation}\label{eq:pq}
p(1/\bar z)=-\frac{\overline{p(z)}}{\bar z^{2n-2}},\qquad
q(1/\bar z)=\frac{\overline{q(z)}}{\bar z^{2n}}.
\end{equation}
Since $u=(r+r^{-1})/2$, we have
$
du=(r^2-1)dr/(2r^2)
$
and
$$
\phi=
\left(\frac{z p(z)}{q(z)}\frac{2r}{r^2-1}\right)du+
\frac{\imag z p(z)}{q(z)}d\theta.
$$
Thus, we can write
$$
\op{Re}(\phi)
=\frac{X(r,\theta)du
+\imag Y(r,\theta)d\theta}{2|q(z)|^2},
$$
where
$$
X(r,\theta):=
\frac{z p(z) 
\overline{q(z)}+\bar z q(z) \overline{p(z)}}{(r-r^{-1})/2},\quad
Y(r,\theta):=
z p(z) 
\overline{q(z)}-\bar z q(z) 
\overline{p(z)}.
$$
Using \eqref{eq:pq}, one can easily verify that
$$
X(1/r,\theta)=\frac{X(r,\theta)}{r^{4n}},\qquad
Y(1/r,\theta)=\frac{Y(r,\theta)}{r^{4n}}.
$$
By Proposition \ref{prop:A},
there exist two polynomials $x(u,\theta)$
and $y(u,\theta)$ in $u$ with parameter $\theta$
such that
$$
X(r,\theta)=r^{2n}x(u,\theta),\qquad
Y(r,\theta)=r^{2n}y(u,\theta).
$$
Hence we have the expression
$$
\op{Re}(\phi)
=\frac{x(u,\theta)du+y(u,\theta)d\theta}
{2\times 4^n \prod_{j=0}^{2n-1}
\biggl(u-\cos(\theta-\alpha_j)\biggr)^{2}}.
$$
By Lemma \ref{lem:interval}, the denominator of
$\op{Re}(\phi)$ does not vanish on
$\Omega_{\alpha_{0},....,\alpha_{2n-1}}$,
which proves the assertion.
\end{proof}

Thus, the three closed forms $\op{Re}(\phi_k)$ 
($k=0,1,2$) are well-defined on a simply 
connected $2$-manifold
$\Omega_{\alpha_{0},....,\alpha_{2n-1}}$, and
the following assertion immediately follows:

\begin{corollary}\label{prop:extf}
Let $f$ be a Kobayashi surface
associated to the Weierstrass data 
as in \eqref{eq:maxfld}.
Then 
$
\tilde f:=f\circ \iota:\hat\Omega_n\cup\{p_\infty\}\to \R^3_1
$ 
can be analytically extended to the domain
$\Omega_{\alpha_{0},....,\alpha_{2n-1}}$.
\end{corollary}

From now on, we consider $\tilde f$ as a map defined on
$\Omega_{\alpha_{0},....,\alpha_{2n-1}}$. 

\begin{proposition}\label{prop;explicit}
Let $f$ be a Kobayashi surface of
principal type associated to the 
Weierstrass data as
in \eqref{eq:maxfld2}.
Suppose that 
$\alpha_0,...,\alpha_{2n-1}$ are all
distinct.
Then 
$\tilde f=(\tilde x_0,\tilde x_1,\tilde x_2)$
has the following expressions:
\begin{equation}\label{eq:map}
\begin{aligned}
\tilde x_0 &= -\sum_{j=0}^{2n-1}
             A_j
             \log\biggl(u-
              \cos(\theta-\alpha_j)\biggr),
            \\
x_1 &= \hphantom{-}\sum_{j=0}^{2n-1}
             A_j \biggl(\cos(n-1)\alpha_j\biggr ) 
             \log\biggl(u-
              \cos(\theta-\alpha_j)\biggr),\\
x_2 &= \hphantom{-}\sum_{j=0}^{2n-1}
             A_j \biggl (\sin(n-1)\alpha_j\biggr )
             \log\biggl(u-
              \cos(\theta-\alpha_j)\biggr),
\end{aligned}
\end{equation}
where
\begin{equation}\label{eq:Ak}
   A_j : = \frac{(-1)^{n+1}}{2^{2n-1}}
           \left(
	      \prod_{i\in \{0,...,2n-1\}\setminus\{j\}}
\sin\frac{\alpha_j-\alpha_i}{2}
	   \right)^{-1}\in \R\setminus\{0\}, 
\qquad (j=0,\dots,2n-1).
\end{equation}
\end{proposition}

\begin{proof}
We use the expression \eqref{eq:phi}.
Since $\alpha_0,...,\alpha_{2n-1}$ are all
distinct, we have the following
partial fractional decomposition:
\begin{equation}\label{eq:phi2}
  \phi_k =  \sum_{j=0}^{2n-1}
  \frac{B_{k,j}}{z-e^{\imag\alpha_j}}\,dz
\qquad (k=0,1,2),
\end{equation}
where $B_{k,j}\in \R$ ($j=0,...,2n-1$) 
is the residue of 
the meromorphic $1$-form $\phi_k$
on $S^2$
at $z=e^{\imag\alpha_j}$.
By a straightforward calculation,
we have 
$$
B_{0,j}=-2A_j,\quad B_{1,j}=2A_j \cos (n-1)\alpha_j,
\quad
B_{2,j}=2A_j \sin (n-1)\alpha_j.
$$
The residue formula yields that
\begin{equation}\label{eq:residue}
  \sum_{j=0}^{2n-1} A_j =
  \sum_{j=0}^{2n-1} A_j\cos(n-1)\alpha_j =
  \sum_{j=0}^{2n-1} A_j\sin(n-1)\alpha_j =0.
\end{equation}
By integrating 
\eqref{eq:phi2}, we get
\begin{align*}
  x_0 &=\Re\int \phi_0 
       = -2 \Re \sum_{j=0}^{2n-1}A_j \log(z-e^{\imag\alpha_j})
       = - \sum_{j=0}^{2n-1}A_j \log |z-e^{\imag\alpha_j}|^2\\
      &= - \sum_{j=0}^{2n-1}A_j \log 
            \left[2 r\left(\frac{1}{2}\left(r+\frac{1}{r}\right)-
  \cos(\theta-\alpha_j)\right)\right]\\
      &= - \sum_{j=0}^{2n-1}A_j \log 
          \left(\frac{1}{2}\left(r+\frac{1}{r}\right)-  
            \cos(\theta-\alpha_j)\right)
         - \log (2 r) \sum_{j=0}^{2n-1}A_j \\
      &= - \sum_{j=0}^{2n-1}A_j \log 
           \left(\frac{1}{2}\left(r+\frac{1}{r}\right)-  
            \cos(\theta-\alpha_j)\right).
 \end{align*}
Similarly, we have
 \begin{align*}
  x_1 & = 2 \Re\sum_{j=0}^{2n-1}A_j 
\biggl(\cos(n-1)  \alpha_j\biggr) \log (z-e^{\imag\alpha_j})\\
      &=  \sum_{j=0}^{2n-1}A_j 
\biggl(\cos(n-1)\alpha_j\biggr)\log 
           \left(\frac{1}{2}\left(r+\frac{1}{r}\right)-  
            \cos(\theta-\alpha_j)\right),\\
  x_2       &=  \sum_{j=0}^{2n-1}A_j 
\biggl(\sin(n-1)\alpha_j\biggr)\log 
           \left(\frac{1}{2}\left(r+\frac{1}{r}\right)-  
            \cos(\theta-\alpha_j)\right).
 \end{align*}
Since $u=(r+r^{-1})/2$, we get the assertion.
\end{proof}

\section{Entire graphs induced by 
Kobayashi surface}\label{sec3}

Let $f$ be a Kobayashi surface
associated to the Weierstrass data 
as in \eqref{eq:maxfld}.
We denote by 
$$
\tilde f=(\tilde x_0,\tilde x_1,\tilde x_2)
:\Omega_{\alpha_{0},....,\alpha_{2n-1}}\to \R^3_1
$$ 
its analytic extension.
We firstly consider the principal case:

\begin{lemma}\label{lem:first}
\red{Let $\tilde f=(\tilde x_1,\tilde x_2,\tilde x_3)$ be the 
analytic extension of the 
Kobayashi surface $f$ of
principal type associated to the 
Weierstrass data as in \eqref{eq:maxfld2}.
Then the map
$
(\tilde x_1,\tilde x_2):
\Omega_{\alpha_{0},....,\alpha_{2n-1}}
\to \R^2
$
is an immersion
if and only if
\begin{equation}\label{cond:1}
|\alpha_j-\alpha_{j+1}|\le \frac{\pi}{n-1}
\qquad (j=0,...,2n-1),
\end{equation}
where $\alpha_{2n}:=2\pi$ and
$\alpha_0,...,\alpha_{2n-1}$ are not necessarily distinct. }
\end{lemma}

\begin{proof}
We can use the technique given in
\cite{FKKRUY}, that is, we use the identity
 \begin{equation}\label{eq:jacobian}
    \frac{\partial(\tilde x_i,\tilde x_j)}{\partial (u,\theta)}
    = \frac{-\imag r^3}{r^2-1}
      \begin{vmatrix}
       \hat \phi_i & \overline{\hat \phi_i}\\
       \hat \phi_j & \overline{\hat \phi_j}
      \end{vmatrix},
 \end{equation}
where $\phi_k=\hat
\phi_k \,dz$ ($k=0,1,2$)．
On the $z$-plane, it holds that 
 \begin{align*}
  &\frac{\partial(\tilde x_1,\tilde x_2)}{\partial(u,\theta)}
  =\frac{-\imag r^3}{r^2-1}\cdot
    \begin{vmatrix}
     \hat \phi_1 & \overline{\hat \phi_1} \\
     \hat \phi_2 & \overline{\hat \phi_2}
    \end{vmatrix}=
   \frac{-\imag r^3}{r^2-1}\cdot
\frac{    \begin{vmatrix}
       \imag \Lambda(1+z^{2n-2}) & -\imag \bar \Lambda
(1+\bar z^{2n-2}) \\
       -\Lambda(1-z^{2n-2}) & -\bar \Lambda(1-\bar z^{2n-2})
    \end{vmatrix}}{%
         \left|\displaystyle\prod_{j=0}^{2n-1}
(z-e^{\imag\alpha_j})^2\right|}
\\
   \phantom{aaa}&= \frac{-r^3}{r^2-1}\cdot
      \frac{2-2 |z|^{4 n-4} }{%
              2^{2n} r^{2n}
       \displaystyle\prod_{j=0}^{2n-1}
        \biggl(u-\cos(\theta-\alpha_j)\biggr)^{2}}
=\frac{U_{2n-3}(u)}{
     2^{2n-1}
       \displaystyle\prod_{j=0}^{2n-1}
        \biggl(u-\cos(\theta-\alpha_j)\biggr)^{2}},
\end{align*}
where $u:=(r+r^{-1})/2$, $\Lambda$ is as in
\eqref{eq:beta},
and
$U_{2n-3}(u)$ is the Chebyshev polynomial of
the second kind of degree $2n-3$, and
we have used the second identity of
\eqref{eq:id}.
By the real analyticity of $\tilde x_1,\tilde x_2$,
the above identity holds on 
$\Omega_{\alpha_{0},....,\alpha_{2n-1}}$.
By \eqref{cond:1} and Corollary \ref{cor:Omega2}, we have 
 \[
     u > \min_{j=0,...,2n-1} \cos\frac{\alpha_{j+1}-\alpha_{j}}{2}
       \ge \cos \frac{\pi}{2(n-1)}
       >\cos \frac{\pi}{2n-3}.
 \]
By \cite[Proposition A.3]{FKKRUY},
$U_{2n-3}(u)$ is monotone increasing
on $[\cos \frac{\pi}{2n-3},\infty)$
and 
\[
U_{2n-3}(u)> U_{2n-3}\left(\cos \frac{\pi}{2(n-1)}\right)=0,
 \]
which implies that
$
\partial(\tilde x_1,\tilde x_2)/{\partial(u,\theta)} \neq 0
$
on $\Omega_{\alpha_{0},....,\alpha_{2n-1}}$,
and we can conclude that
$(u,\theta)\mapsto (\tilde x_1,\tilde x_2)$ 
is an immersion.

\red{On the other hand,
if $(\alpha_0,...,\alpha_{2n-1})$
does not satisfy the condition
\eqref{cond:1}}, then we may set
$\alpha_0=0$ \red{(cf. Remark \ref{eq:rotation})}
and $\alpha_1>{\pi}/(n-1)$.
We set
$$
u_0:=\cos \frac{\pi}{2(n-1)},\qquad
\theta_0:=\frac{\alpha_1}2.
$$
By Corollary \ref{cor:Omega},
we have
$(u_0,\theta_0)\in \Omega_{\alpha_0,...,\alpha_{2n-1}}$.
Since $U_{2n-3}(u_0)=0$, we have
$$
   \left. \frac{\partial(\tilde x_1,\tilde x_2)}
{\partial (u,\theta)}\right|_{(u,\theta)=(u_0,\theta_0)}=0.
$$
In particular, the induced map
$(\tilde x_1,\tilde x_2)$ of the Kobayashi surface
$f$ is not an immersion at $(u_0,\theta_0)$.
So the condition \eqref{cond:1} is sharp.
\end{proof}

\begin{corollary}
Let $f$ be a Kobayashi surface 
of general type associated 
to the Weierstrass data as
in \eqref{eq:maxfld}.
Suppose that the inequality
\eqref{cond:1} holds. 
If $b_1,...,b_{n-1}$ are sufficiently close to $0$,
then the map
$
(\tilde x_1,\tilde x_2):
\Omega_{\alpha_{0},....,\alpha_{2n-1}}
\to \R^2
$
is an immersion.
\end{corollary}

\begin{proof}
Applying the same technique as in
the proof of Lemma \ref{lem:first},
we have that
\begin{align*}
\frac{\partial(\tilde x_1,\tilde x_2)}{\partial(u,\theta)}
&=
-4^{-n}\frac{X(r,\theta)/r^{2n-2}}{(r-r^{-1})/2}{\displaystyle
\prod_{j=0}^{2n-1}
\left(u-\cos(\theta-\alpha_j)\right)^{-2}},
\end{align*}
where 
we used the equations
\eqref{eq:jacobian}, \eqref{eq:phi0}, \eqref{eq:Q3}
and \eqref{eq:q}, 
and

$$
X(r,\theta)=
\prod_{i=1}^{n-1}|\bar b_iz-1|^4
-\prod_{i=1}^{n-1}|z- b_i|^4
$$
can be considered as a polynomial in $r$
with parameter $\theta$
of degree $4(n-1)$, since $b_1,...,b_{n-1}$ lie in
the unit disk.
Moreover, it holds that
$$
X(1/r,\theta)=-\frac{1}{r^{4n-4}}X(r,\theta).
$$
Hence, $X$ is an anti-self-reciprocal 
polynomial in $r$ of degree $4n-4$.
By Proposition \ref{prop:A} in the appendix,
there exists a polynomial $Y(u,\theta)$
in $u$ of degree $2n-3$ such that
$$
X(r,\theta)=r^{2n-2}\left(
\frac{r-r^{-1}}{2}\right)Y(u,\theta).
$$
Thus, we get the identity
$$
  \frac{\partial(\tilde x_1,\tilde x_2)}{\partial(u,\theta)}
=4^{-n}Y(u,\theta)\displaystyle\prod_{j=0}^{2n-1}
        \left(u-\cos(\theta-\alpha_j)\right)^{-2}.
$$
Since the degree of $Y(u,\theta)$ is the same 
as that of $U_{2n-3}(u)$.
By the continuity of the roots of 
the equation $Y(u,\theta)=0$
with respect to the parameters $b_1,...,b_{n-1}$,
we can conclude that $(\tilde x_1,\tilde x_2)$
is an immersion on 
$\Omega_{\alpha_{0},....,\alpha_{2n-1}}$.
\end{proof}

We next discuss the properness of
the map $(\tilde x_1,\tilde x_2)$:

\begin{proposition}\label{thm:proper}
Let $f$ be a Kobayashi surface 
of principal type
associated to the Weierstrass data as
in \eqref{eq:maxfld2}.
Suppose that the inequality
\eqref{cond:1} holds
and also that $\alpha_0,...,\alpha_{2n-1}$ are
distinct.
Then
$
(\tilde x_1,\tilde x_2):
\Omega_{\alpha_0,...,\alpha_{2n-1}}\to \R^2
$
is a proper immersion.
\end{proposition}

\begin{proof}
Let
$\{(u_m,\theta_m)\}_{m=1,2,3,...}$
be a sequence  on $\Omega_{\alpha_0,...,\alpha_{2n-1}}$
converging to a point  
$(u_{\infty},\theta_{\infty})$
on the boundary of
$\Omega_{\alpha_0,...,\alpha_{2n-1}}$
in the $(u,\theta)$-plane.
Then it is sufficient to show that
$(\tilde x_1(u_m,\theta_m),\tilde x_2(u_m,\theta_m))$
diverges.
We first consider the case
$\theta_{\infty}\ne \gamma_j$ ($j=0,1,\dots,2n-1$),
where $\gamma_j$ is defined
in \eqref{eq:beta-interval}.
Without loss of generality,
we may assume that
$\theta_{\infty}\in I_1
\setminus\{\gamma_{0},\gamma_1\}$,
where $I_1$ is the closed interval as in
\eqref{eq:beta-interval}. 
By Lemma \ref{lem:interval},
we have
$$
    u_{\infty}-\cos(\theta_{\infty}-\alpha_1) =0,\quad
    u_{\infty}-\cos(\theta_{\infty}-\alpha_j) > 0\qquad (j\neq 1).
$$
By \eqref{eq:map}, we have
$$
\pmt{\tilde x_1\\
\tilde x_2}
=
A_1\pmt{
\cos(n-1)\alpha_1 \\
 \sin(n-1)\alpha_1
}
 X_1+\mbox{(a bounded term)},
$$
where
\begin{equation}\label{eq:Xl}
X_j= \log\biggl(u-
              \cos(\theta-\alpha_j)\biggr)
\qquad (j=0,...,2n-1).
\end{equation}
Thus,
$(\tilde x_1(u_m,\theta_m),
\tilde x_2(u_m,\theta_m))$
is unbounded when $m\to \infty$,
since $A_1\ne 0$.  
We next consider the case that
$\theta_\infty=\gamma_j$ for some $j$.
Without loss of generality, we may assume that
$j=1$ and $\theta_\infty=\gamma_1$.
In this case
$X_1$ and $X_2$
both tend to $-\infty$.
We now assume that
$$
\alpha_2-\alpha_{1}<\frac{\pi}{n-1}.
$$
Then  we have
$$
     \pmt{\tilde x_1 \\ \tilde x_2}=
   \begin{pmatrix}
     A_1\cos(n-1)\alpha_1 & A_2\cos(n-1)\alpha_{2} \\
     A_1\sin(n-1)\alpha_1 & A_2\sin(n-1)\alpha_{2}
    \end{pmatrix}
    \begin{pmatrix}
     X_1 \\
     X_2
    \end{pmatrix} + \mbox{(a bounded term)},
$$
which implies that $(\tilde x_1(u_m,\theta_m),
\tilde x_2(u_m,\theta_m))$
is unbounded when $m\to \infty$.
Here, we used the fact that 
 \begin{equation}
\begin{vmatrix}\label{eq:nonzero}
     A_1\cos(n-1)\alpha_1 & A_2\cos(n-1)\alpha_{2} \\
     A_1\sin(n-1)\alpha_1 & A_2\sin(n-1)\alpha_{2}
    \end{vmatrix}
= A_1A_2\sin\biggl((n-1)(\alpha_{2}-\alpha_1)\biggr)\ne 0.
  \end{equation}
Finally, we consider the case
$
\alpha_2-\alpha_{1}={\pi}/(n-1).
$
In this case, we have
$$
\cos (n-1)\alpha_2=-\cos (n-1)\alpha_1,
$$
and
\begin{align*}
\tilde x_1
&=A_1 X_1 \cos (n-1)\alpha_1
+A_2 X_2 \cos (n-1)\alpha_2+\mbox{(a bounded term)}\\
&=(A_1 X_1-A_2X_2) \cos (n-1)\alpha_1
+\mbox{(a bounded term)}.
\end{align*}
\red{Similarly, it holds that
$$
\tilde x_2
=(A_1 X_1-A_2X_2) \sin (n-1)\alpha_1
+\mbox{(a bounded term)}.
$$}
Regarding the fact that the sign of 
$$
\prod_{j\in \{0,...,2n-1\}
\setminus \{i\}} \sin \frac{\alpha_j-\alpha_i}2
$$
is equal to $(-1)^{i+1}$,
we can conclude that $A_1A_2<0$.
Hence
$$
\lim_{m\to \infty}|A_1 X_1-A_2X_2|=\infty.
$$
\red{Since $\cos (n-1)\alpha_1$ and 
$\sin (n-1)\alpha_1$ do not vanish 
at the same time,
either $\tilde x_1(u_m,\theta_m)$ or 
$\tilde x_2(u_m,\theta_m)$ diverges as $m\to \infty$,
which proves the assertion.}
\end{proof}

\begin{corollary}\label{cor:proper}
Let $f$ be a Kobayashi surface 
of general type associated to 
the Weierstrass data as in \eqref{eq:maxfld}.
Suppose that the inequality
\begin{equation}\label{cond:2}
|\alpha_j-\alpha_{j+1}|< \frac{\pi}{n-1}
\qquad (j=0,...,2n-1)
\end{equation}
holds, and also that
$\alpha_0,...,\alpha_{2n-1}$ are distinct.
If $b_1,...,b_{n-1}$ are 
sufficiently close to $0$, then the map
$
(\tilde x_1,\tilde x_2)
:\Omega_{\alpha_0,...,\alpha_{2n-1}}\to \R^2
$
is a proper immersion.
\end{corollary}

\begin{proof}
Regarding the expressions 
\eqref{eq:phi0} and using
the assumption that
$\alpha_0,...,\alpha_{2n-1}$ are distinct,
we can write
$$
\phi_k=\sum_{j=0}^{2n-1}\frac{B_{k,j}}{z-e^{\imag \alpha_j}}dz
\qquad (k=0,1,2),
$$
where $B_{k,j}$ are
real numbers depending on the coefficients of the polynomials
$p_k(z)$ given in \eqref{eq:Q2}, \eqref{eq:Q3}.
In particular, $B_{k,j}$ depend continuously
on the $n-1$ parameters $b_1,...,b_{n-1}$.
Since
$$
\tilde x_k=\frac12 \sum_{j=0}^{2n-1} B_{k,j}
\log\biggl(u-\cos(\theta-\alpha_j)\biggr)
\qquad (k=0,1,2),
$$
one can prove the assertion 
using the same argument as in the proof of
Proposition~\ref{thm:proper},
using the identity
$$
     \pmt{\tilde x_1 \\ \tilde x_2}=\frac12
   \begin{pmatrix}
     B_{1,1} & B_{1,2} \\
     B_{2,1} & B_{2,2}
    \end{pmatrix}
    \begin{pmatrix}
     \log\bigl(u-\cos(\theta-\alpha_1)\bigr) \\
    \log\bigl(u-\cos(\theta-\alpha_2)\bigr)
    \end{pmatrix} + \mbox{(a bounded term)},
$$
and the limit formula
$$
\lim_{(b_1,...,b_{n-1})\to \vect 0}
\frac12\pmt{B_{1,1},B_{1,2}\\ B_{2,1},B_{2,2}}
=
\begin{pmatrix}
     A_1\cos(n-1)\alpha_1 &  A_2\cos(n-1)\alpha_{2} \\
 A_1\sin(n-1)\alpha_1  & A_2\sin(n-1)\alpha_{2}
    \end{pmatrix}.
$$
\end{proof}

We now get the following result:

\begin{theorem}
Let $f$ be a Kobayashi surface of
principal type associated to the 
Weierstrass data as in \eqref{eq:maxfld2}.
Suppose that $(\alpha_{2n}:=2\pi)$
\begin{equation}
|\alpha_j-\alpha_{j+1}|\le \frac{\pi}{n-1}
\qquad(j=0,....,2n-1),
\end{equation}
and that $\alpha_0,...,\alpha_{2n-1}$ are
distinct.
We let $\tilde f:\Omega_{\alpha_{0},...,\alpha_{2n-1}}
\to \R^3_1$ denote its analytic extension. 
 Then the map
$
(\tilde x_1,\tilde x_2):\Omega_{\alpha_{0},...,\alpha_{2n-1}}
\to \R^2
$
is a diffeomorphism. 
In particular, the image of $\tilde f$
gives a zero-mean curvature entire graph of
mixed type.
\end{theorem}

\begin{proof}
This theorem follows immediately from 
Proposition \ref{thm:proper},
applying the fact that
proper \red{immersions} between the same dimensional
manifolds are diffeomorphisms \red{under the
assumption that the target space is
simply-connected} 
(cf. \cite[Corollary]{H}).
\end{proof}

Similarly, the following assertion
immediately follows
from Corollary \ref{cor:proper}.

\begin{theorem}\label{thm:gen}
Let $f$ be a Kobayashi surface of 
general type associated to the Weierstrass data as
in \eqref{eq:maxfld}.
Suppose that the inequality
\begin{equation}\label{eq:condA}
|\alpha_j-\alpha_{j+1}|< \frac{\pi}{n-1}
\qquad (j=0,...,2n-1),
\end{equation}
and also that $\alpha_0,...,\alpha_{2n-1}$ are
distinct. 
We let $\tilde f:\Omega_{\alpha_{0},...,\alpha_{2n-1}}
\to \R^3_1$ denote its analytic extension. 
If $b_1,...,b_{n-1}$ are 
sufficiently close to $0$ then the map
$$
(\tilde x_1,\tilde x_2):
\Omega_{\alpha_{0},...,\alpha_{2n-1}}
\to \R^2
$$
is a diffeomorphism.
In particular, the image of $\tilde f$
gives a zero-mean curvature entire graph of
mixed type.
As a consequence, we get a 
 $(4n-7)$-parameter family
of zero-mean curvature entire graphs
up to congruence and homothety.
\end{theorem}

\begin{remark}
The assumption in the theorem that
$b_1,...,b_{n-1}$ are 
sufficiently close to $0$ 
is crucial, and in fact, for larger
$b_1,...,b_{n-1}$, the corresponding
Kobayashi surface might not give 
an entire graph.
In fact, let $f_b$ be 
the Kobayashi surface of order $4$
with angular data 
$\alpha_j=\pi j/4$ ($j=0,1,...,7$)
whose Weierstrass data is
$$
g=z^2\frac{z-b}{1-\overline{b}z},\qquad
\omega=
\frac{(1-\overline{b}z)^2}
{z^8-1}dz.
$$
We know that the analytic extension of 
the image of $f_b$
gives an entire graph if $|b|$ is sufficiently
small. 
However, for example, if we set $b=-0.75$,
$f_b$ appears to have self-intersections.
(See Figure \ref{fig:beq0pt9}).
\end{remark}

\begin{figure}
 \centering
 \includegraphics[width=0.5\textwidth]{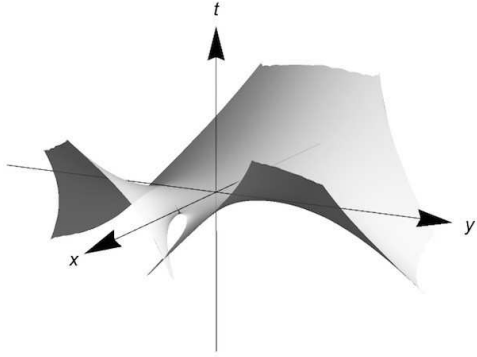}
 \caption{
The figure of $f_b$ for $b=-0.75$} 
\label{fig:beq0pt9}
\end{figure}

\section{The condition for $\tilde f$
of principal type to be immersed}
\label{sec:4}

In this section, we consider 
Kobayashi surfaces of principal type
with an angular data $(\alpha_0,...,\alpha_{2n-1})$
satisfying the weaker condition (cf. \eqref{cond:1})
\begin{equation}\label{condB}
|\alpha_j-\alpha_{j+1}|\le  \frac{2\pi}{n-1}
\qquad (j=0,...,2n-1),
\end{equation}
where $\alpha_{2n}:=2\pi$. 
Under this condition, we prove the following:

\begin{lemma}\label{lem:second}
Let $f$ be a Kobayashi surface of
principal type associated to the 
Weierstrass data as
in \eqref{eq:maxfld2}.
\red{Then 
$
\tilde f:\Omega_{\alpha_{0},....,\alpha_{2n-1}}
\to \R^3_1
$
is an immersion
if and only if
$(\alpha_0,...,\alpha_{2n-1})$
satisfies \eqref{condB}.
$($Here 
$\alpha_0,...,\alpha_{2n-1}$ 
are not necessarily distinct.$)$ }
\end{lemma}

\begin{proof}
Using \eqref{eq:jacobian}, we have
(see \eqref{eq:beta} for the definition of $\Lambda$)
\allowdisplaybreaks{%
 \begin{align*}
  \frac{\partial(\tilde x_0,\tilde x_1)}{\partial(u,\theta)}
  &=\frac{-\imag r^3}{r^2-1}\cdot
    \begin{vmatrix}
     \hat\phi_0 & \overline{ \hat\phi_0} \\
      \hat\phi_1 & \overline{ \hat\phi_1}
    \end{vmatrix} 
=
   \frac{-\imag r^3}{r^2-1}\cdot
\frac{    \begin{vmatrix}
      -2\imag \Lambda z^{n-1} & 2\imag \bar \Lambda
\bar z^{n-1}\\
       \imag \Lambda(1+z^{2n-2}) & -\imag 
\bar  \Lambda(1+\bar z^{2n-2})
    \end{vmatrix}}{%
         \left|\displaystyle\prod_{j=0}^{2n-1}
(z-e^{\imag\alpha_j})^2\right|}
\\
   &= \frac{2\imag r^3}{r^2-1}\cdot
      \frac{z^{n-1}+z^{n-1}\bar z^{2n-2}-\bar z^{n-1}-\bar
          z^{n-1}z^{2n-2}}{%
              2^{2n} r^{2n}
       \displaystyle\prod_{j=0}^{2n-1}
        \biggl(u-\cos(\theta-\alpha_j)\biggr)^{2}}\\
   &= \frac{4(r^{n-1}-r^{-n+1})}{r-r^{-1}}\cdot
      \frac{\sin (n-1)\theta}{
              2^{2n}
       \displaystyle\prod_{j=0}^{2n-1}
        \biggl(u-\cos(\theta-\alpha_j)\biggr)^{2}}\\
&=
    \frac{U_{n-2}(u)\sin (n-1)\theta}{
              2^{2n-2}
       \displaystyle\prod_{j=0}^{2n-1}
        \biggl(u-\cos(\theta-\alpha_j)\biggr)^{2}}
 \end{align*}}
and
 \begin{align*}
  \frac{\partial(\tilde x_0,\tilde x_2)}{\partial(u,\theta)}
  &=\frac{-\imag r^3}{r^2-1}\cdot
    \begin{vmatrix}
     \hat \phi_0 & \overline{\hat \phi_0} \\
     \hat \phi_2 & \overline{\hat \phi_2}
    \end{vmatrix}=
   \frac{-\imag r^3}{r^2-1}\cdot
\frac{    \begin{vmatrix}
      -2\imag \Lambda z^{n-1} & 2\imag \bar \Lambda
\bar z^{n-1}\\
       -\Lambda(1-z^{2n-2}) & - \bar \Lambda
(1-\bar z^{2n-2})
    \end{vmatrix}}{%
         \left|\displaystyle\prod_{j=0}^{2n-1}
(z-e^{\imag\alpha_j})^2\right|}
\\
&=
    \frac{-U_{n-2}(u)\cos (n-1)\theta}{
              2^{2n-2}
       \displaystyle\prod_{j=0}^{2n-1}
        \biggl(u-\cos(\theta-\alpha_j)\biggr)^{2}},
 \end{align*}
where $U_{n-2}$ is
the Chebyshev polynomial of the
second kind of degree $n-2$.
Since $|\alpha_{j}-\alpha_{j+1}|\leq {2\pi}/(n-1)$,
Corollary \ref{cor:Omega2} yields that
 \[
     u > \cos\frac{\pi}{n-1}>\cos\frac{\pi}{n-2}.
 \] 
By \cite[Proposition A.3]{FKKRUY},
$U_{n-2}(u)$ is monotone increasing
on $[\cos \frac{\pi}{n-2},\infty)$
and
 \[
     U_{n-2}(u)> U_{n-2}\left(\cos\frac{\pi}{n-1}\right)
               = \frac{\sin\pi}{\sin(\pi/(n-1))}=0
 \]
holds.
In particular, $U_{n-2}(u)> 0$
on  $\Omega_{\alpha_{0},....,\alpha_{2n-1}}$.
Since
	\[
	   \frac{\partial(\tilde x_0,\tilde x_1)}{\partial
	(u,\theta)},\qquad
	   \frac{\partial(\tilde x_0,\tilde x_2)}{\partial (u,\theta)}
	\]
do not vanish simultaneously, we can
conclude that
$\tilde f$ is an immersion.

\red{On the other hand, if} $(\alpha_0,...,\alpha_{2n-1})$
does not satisfy the condition
\eqref{condB}, then $\tilde f$
is not an immersion.
In fact, in this case, we may set
$\alpha_0=0$ and $\alpha_{1}>{2\pi}/(n-1)$.
We set
$$
u_0:=\cos\left(\frac{\pi}{n-1}\right),\qquad
\theta_0:=\frac{\alpha_1}2.
$$
By Corollary \ref{cor:Omega},
$(u_0,\theta_0)\in \Omega_{\alpha_0,...,\alpha_{2n-1}}$.
Since $U_{n-1}(u_0)=0$, we have
$$
\left.\frac{\partial(\tilde x_0,\tilde x_1)}
{\partial (u,\theta)}\right|_{(u,\theta)= (u_0,\theta_0)}
=
\left.\frac{\partial(\tilde x_0,\tilde x_2)}
{\partial (u,\theta)}\right|_{(u,\theta)= (u_0,\theta_0)}
=0.
$$
On the other hand,
the identity 
$\sin (2n-2)\theta=2\sin (n-1)\theta \cos (n-1)\theta$
induces the relation $U_{2n-3}=2T_{n-1}U_{n-2}$
for the Chebyshev polynomials.
In particular,  $U_{2n-3}(u_0)=0$ and
(see the proof of Lemma \ref{lem:first})
$$
\left.\frac{\partial(\tilde x_1,\tilde x_2)}
{\partial (u,\theta)}\right|_{(u,\theta)= (u_0,\theta_0)}=0.
$$
Hence $\tilde f$ is not an immersion at $(u_0,\theta_0)$.
\end{proof}

We next discuss the properness of
the map $\tilde f$:

\begin{proposition}\label{thm:proper2}
Let $f$ be a Kobayashi surface 
of principal type associated to the Weierstrass data as
in \eqref{eq:maxfld2}.
Suppose that the inequalities
\begin{equation}\label{condB2}
|\alpha_j-\alpha_{j+1}|<\frac{2\pi}{n-1}
\qquad (j=0,...,2n-1)
\end{equation}
hold, and also that $\alpha_0,...,\alpha_{2n-1}$ are
distinct.
Then
$
\tilde f:
\Omega_{\alpha_0,...,\alpha_{2n-1}}\to \R^3_1
$
is a proper immersion.
\end{proposition}

\begin{proof}
By the previous lemma, we know that
$\tilde f$ is an immersion. So, we will
show that $\tilde f$ is proper.
Let
$\{(u_m,\theta_m)\}_{m=1,2,3,...}$
be a sequence  
on $\Omega_{\alpha_0,...,\alpha_{2n-1}}$
converging to a point  
$(u_{\infty},\theta_{\infty})$
on the boundary of
$\Omega_{\alpha_0,...,\alpha_{2n-1}}$.
Then it is sufficient to show that
$\tilde f(u_m,\theta_m)$
diverges.
In the case that
$\theta_{\infty}\ne \gamma_j$ ($j=0,1,\dots,2n-1$),
we can conclude that $\tilde f(u_m,\theta_m)$
diverges using the same argument as in the
proof of Proposition \ref{thm:proper},
where $\gamma_j$ is defined
in \eqref{eq:beta-interval}.
(The proof of Proposition \ref{thm:proper}
is given under the assumption that
$\alpha_0,...,\alpha_{2n-1}$ are distinct.)
So it is sufficient to consider the case that
$\theta_\infty=\gamma_j$ for some $j$.
Without loss of generality, we may assume that
$j=1$ and $\theta_\infty=\gamma_1$.
In this case, we have
$$
\pmt{\tilde x_0 \\ \tilde x_1 \\
\tilde x_2}=M
    \begin{pmatrix}
     X_1 \\
     X_2
    \end{pmatrix} + \mbox{(a bounded term)},
$$
where \red{$X_1$ and $X_2$ are}
given in 
\eqref{eq:Xl} and
$$
M:=   \begin{pmatrix}
 A_1 & A_2 \\
 A_1\cos(n-1)\alpha_1 & A_2\cos(n-1)\alpha_2 \\
 A_1\sin(n-1)\alpha_{1} & A_2\sin(n-1)\alpha_{2}
    \end{pmatrix}.
$$
This implies that $\tilde f(u_m,\theta_m)$
is unbounded as $m\to \infty$
\red{if} the rank of the matrix $M$
is $2$. In fact, since $A_1,A_2 \ne 0$ we have
$$
\op{rank}(M)=1+
\op{rank}
\pmt{
\cos(n-1)\alpha_2 - \cos(n-1)\alpha_1 \\
\sin(n-1)\alpha_2 - \sin(n-1)\alpha_1},
$$
and
$$
\pmt{
\cos(n-1)\alpha_2 - \cos(n-1)\alpha_1 \\
\sin(n-1)\alpha_2 - \sin(n-1)\alpha_1}
=
2\sin\left((n-1)\frac{\alpha_2-\alpha_1}{2}\right)
\pmt{-\sin (n-1) \gamma_1 \\ \cos (n-1) \gamma_1}
$$
never vanishes because 
$0<|\alpha_2-\alpha_1|<{2\pi}/(n-1)$.
So we get the assertion. 
\end{proof}

In the authors' numerical experiments,
the following question has naturally arisen:

\medskip
\noindent
{\bf Problem.}
{\it Let $f$ be a Kobayashi surface of
principal type associated to the 
Weierstrass data as
in \eqref{eq:maxfld2}.
Suppose that 
$(\alpha_0,...,\alpha_{2n-1})$
satisfies \eqref{condB}
but not \eqref{eq:condA}.
$($Here $\alpha_0,...,\alpha_{2n-1}$ are not necessarily 
distinct.$)$ 
\red{Then, can one find a suitable condition for
the analytic extension
$
\tilde f:\Omega_{\alpha_{0},....,\alpha_{2n-1}}
\to \R^3_1
$
to be properly embedded?}
}

\begin{figure}[h]
 \centering
 \includegraphics[width=0.5\textwidth]{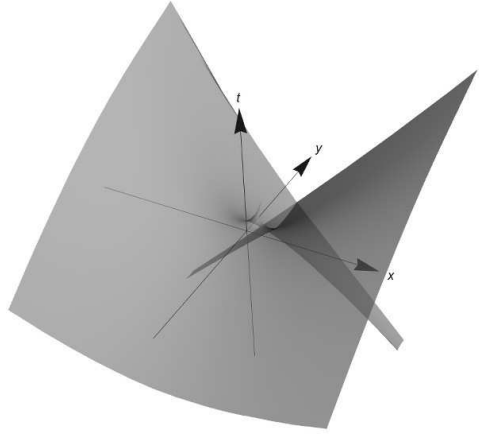}
 \caption{
\red{Kobayashi surface of principal type with angular data $(0,3\pi/4,3\pi/2,5\pi/3,7\pi/4,11\pi/6)$.}} 
\label{fig:newexample}
\end{figure}

\medskip
As a special case, the authors proved 
the proper embeddedness of
 the analytic extension of the Jorge-Meeks type
maximal surface $\mathcal J_n$
as a Kobayashi surface of principal type
of order $n$, in the previous work \cite{FKKRUY}.  
\red{However, if $n=3$, the Kobayashi surface 
of principal type with angular data
$(0,3\pi/4,3\pi/2,5\pi/3,7\pi/4,11\pi/6)$ 
has an immersed analytic extension having self-intersections
(cf. Figure \ref{fig:newexample}).
On the other hand,} for $n=2$, the condition
\eqref{condB} gives no restriction for
the angular data $(\alpha_0,\alpha_1,\alpha_2,\alpha_3)$.
We can prove the following,
which tells us that the problem is affirmative
when $n=2$:

\begin{theorem}
Let $f$ be a Kobayashi surface of
order $2$.
Then its analytic extension
$
\tilde f:\Omega_{\alpha_0,\alpha_1,\alpha_2,\alpha_3}
\to \R^3_1
$
$(\alpha_0=0)$ is properly embedded.
\end{theorem}

\begin{proof}
By Proposition \ref{prop:4}, $f$ is a Kobayashi
surface of principal type.
We may assume that its Weierstrass data is
written as in \eqref{neq2}.

\medskip
We first consider the case that
the angular data $\alpha_0,...,\alpha_3$ are distinct.
By Proposition \ref{thm:proper2}, $\tilde f$
is a proper immersion.
So it is sufficient to show that $\tilde f$
is an injection.
Using \eqref{eq:map}
and \eqref{eq:residue}, we can write
\begin{align*}
\pmt{
\tilde x_0\\
\tilde x_1\\
\tilde x_2
}
&=\pmt{
A_1(X_1-X_0)+A_2(X_2-X_0)+A_3(X_3-X_0)\\
-A_1(X_1-X_0)\cos \alpha_1
-A_2(X_2-X_0)\cos \alpha_2
-A_3(X_3-X_0)\cos \alpha_3\\
-A_1(X_1-X_0)\sin \alpha_1
-A_2(X_2-X_0)\sin \alpha_2
-A_3(X_3-X_0)\sin \alpha_3} \\
&=
R \pmt{Y_1\\ Y_2 \\ Y_3}
\qquad
(Y_j:=X_j-X_0,\,\, j=1,2,3),
\end{align*}
where
$$
X_j:= \log\biggl(u-
              \cos(\theta-\alpha_j)\biggr)
\qquad (j=0,1,2,3)
$$
and
$$
R:=\pmt{A_1 & A_2& A_3\\
-A_1 \cos \alpha_1& -A_2 \cos \alpha_2 
& -A_3 \cos \alpha_3 \\
-A_1 \sin \alpha_1& -A_2 \sin \alpha_2 
& -A_3 \sin \alpha_3 
}.
$$
Since the determinant of $R$ is equal to 
$$
4A_1A_2A_3\sin\frac{\alpha_2-\alpha_1}2
\sin\frac{\alpha_3-\alpha_2}2
\sin\frac{\alpha_3-\alpha_1}2(\ne 0),
$$
the matrix $R$ is regular.
Since
$$
Y_j=X_j-X_0=\log\frac{u-\cos(\theta-\alpha_j)}
{u-\cos \theta}
\qquad (j=1,2,3),
$$
it is sufficient to show that the map
$
\Phi:\Omega_{0,\alpha_1,\alpha_2,\alpha_3}
\to \R^3
$
given by
$$
\Phi (u, \theta):=
\frac1{u-\cos \theta}\biggl(
u-\cos(\theta-\alpha_1),
u-\cos(\theta-\alpha_2),
u-\cos(\theta-\alpha_3)
\biggr)
$$
is injective. It holds that
$$
\Phi (u, \theta)
=2\pmt{
\sin \frac{\alpha_1}2 & 0 & 0\\
0 &\sin \frac{\alpha_2}2 & 0 \\
0 & 0 & \sin \frac{\alpha_3}2 
}
\pmt{
\sin \frac{\alpha_1}2 & -\cos \frac{\alpha_1}2 \\
\sin \frac{\alpha_2}2 &-\cos \frac{\alpha_2}2 \\
\sin \frac{\alpha_3}2 &-\cos \frac{\alpha_3}2 }
\pmt{\xi \\ \eta
}+\pmt{1\\1\\1}
,
$$
where
\begin{equation}\label{eq:xy}
\xi:=
\frac{\cos \theta}{u-\cos \theta}, \qquad
\eta:=\frac{\sin \theta}{u-\cos \theta}. 
\end{equation}
Since
$$
\vmt{
\sin \frac{\alpha_1}2 & -\cos \frac{\alpha_1}2 \\
\sin \frac{\alpha_2}2 &-\cos \frac{\alpha_2}2 }
=
\sin\frac{\alpha_2-\alpha_1}2 \ne 0,
$$
the injectivity of $\Phi$ is
equivalent to the injectivity of the
map
\begin{equation}\label{eq:Psi}
\Psi:\Omega_{0,\alpha_1,\alpha_2,\alpha_3}
\ni 
(u,\theta)
\mapsto 
\frac1{u-\cos \theta}(\cos \theta,\sin \theta)
\in \R^2.
\end{equation}
This follows from the injectivity of the map
$$
\tilde\Psi:\Omega_{0,\alpha_1,\alpha_2,\alpha_3}
\ni 
(u,\theta)
\mapsto 
[u-\cos \theta:\cos \theta:\sin \theta]
\in P^2,
$$
where $P^2$ is the real projective plane.
Suppose that $\tilde\Psi(u_1,\theta_1)
=\tilde\Psi(u_2,\theta_2)$.
Since $u_j-\cos \theta_j>0$ for $(j=1,2)$,
we have
$$
[\cos \theta_1:\sin \theta_1]=
[\cos \theta_2:\sin \theta_2]
$$
and so
$
\theta_2-\theta_1\equiv 0 \mod \pi,
$
that is, $\cos \theta_2=\pm \cos \theta_1$.
However $\cos \theta_2\ne -\cos \theta_1$
because of the fact that
$\Omega_{0,\alpha_1,\alpha_2,\alpha_3}$
is contained in the upper half of 
the $(u,\theta)$-plane.
This implies the injectivity of $\tilde \Psi$.

\medskip
We next consider the case that
$$
\alpha_0=\alpha_1=0,\quad
\alpha_2=\alpha,\quad 
\alpha_3=\beta 
\qquad (0<\alpha<\beta<2\pi).
$$
We can write (cf. \eqref{eq:phi})
$$
\phi_0=\left[\frac{\imag B'}{(z-1)^2}+
A'_1\left(\frac1{z-e^{\imag \alpha}}
-\frac1{z-1}\right)+
A'_2\left(\frac1{z-e^{\imag \beta}}
-\frac1{z-1}\right)\right]dz.
$$
In particular, $A'_1$ and
$A'_2$ are residues of $\phi_0$ at $z=e^{\imag \alpha}$
and $z=e^{\imag \beta}$, respectively.
Then
$$
B':=\frac{1}{2\sin \frac{\alpha}2 \sin \frac{\beta}2},
\quad
A'_1:=\frac{1}{4\sin^2 \frac{\alpha}2 \sin \frac{\alpha-\beta}2},
\quad
A'_2:=\frac{1}{4\sin^2 \frac{\beta}2 \sin \frac{\beta-\alpha}2}
$$
hold.
Moreover, we have that
\begin{align*}
\phi_1&=\left[\frac{-\imag B'}{(z-1)^2}-A'_1\cos \alpha 
\left(\frac1{z-e^{\imag \alpha}}
-\frac1{z-1}\right)-A'_2\cos \beta
\left(\frac1{z-e^{\imag \beta}}
-\frac1{z-1}\right)\right]dz,\\
\phi_2&=\left[
-A'_1\sin \alpha 
\left(\frac1{z-e^{\imag \alpha}}
-\frac1{z-1}\right)-A'_2\sin \beta
\left(\frac1{z-e^{\imag \beta}}
-\frac1{z-1}\right)\right]dz.
\end{align*}
Integrating them, taking the real parts, and
replacing parameter $r$ by $u$, we have that
$$
\pmt{\tilde x_0\\
\tilde x_1\\
\tilde x_2}
=\frac12
\pmt{-B' & A'_1 & A'_2 \\
        B' & -A'_1\cos \alpha & -A'_2\cos \beta \\
        0 & -A'_1\sin \alpha & -A'_2\sin \beta 
}
\pmt{X'_0\\ X'_1\\ X'_2},
$$
where
$$
X'_0:=\frac{\sin \theta}{u-\cos \theta},\quad
X'_1:=\log \left(
\frac{u-\cos(\theta-\alpha)}{u-\cos \theta}
\right),\quad
X'_2:=
\log \left(\frac{u-\cos(\theta-\beta)}
{u-\cos \theta}\right).
$$
Since
$$
\vmt{-B' & A'_1 & A'_2 \\
        B' & -A'_1\cos \alpha & -A'_2\cos \beta \\
        0 & -A'_1\sin \alpha & -A'_2\sin \beta 
}
=B'A'_1A'_2 \sin\frac{\alpha}{2}\sin\frac{\beta}{2}
\sin\frac{\alpha-\beta}{2},
$$
it is sufficient to show that
the map $(u,\theta)\mapsto (X'_0,X'_1,X'_2)$
is a proper embedding.
The properness follows from
the fact that at least one of 
the three functions $X'_0,X'_1,X'_2$
diverges along the
sequence $\{(u_m,\theta_m)\}$ converging
to the boundary of $\Omega_{0,0,\alpha,\beta}$.
We see that
$$
\pmt{X'_0 \\ e^{X'_1}-1 \\ e^{X'_2}-1}
=
\pmt{0 & 1 \\
1-\cos \alpha & -\sin \alpha \\
1-\cos \beta & -\sin \beta}
\pmt{\xi \\ \eta}
,
$$
where $\xi,\eta$ are the functions
given in \eqref{eq:xy}.
Thus, the embeddedness reduces to 
the injectivity of the map $\Psi$
(cf. \eqref{eq:Psi}).

\medskip
Similarly, we consider the case that
$$
\alpha_0=\alpha_1=0,\quad
\alpha_2=\alpha_3=\alpha 
\qquad (0<\alpha<2\pi).
$$
Then
$$
\pmt{\tilde x_0\\
\tilde x_1\\
\tilde x_2}
=
\pmt{-A'' & -A'' & -A''B'' \\
        A''  &   A''\cos \alpha & A''B'' \\
         0  & B''/2  & A''
}
\pmt{X''_0\\ X''_1\\ X''_2},
$$
where
$$
A'':=\frac{1}{4\sin^2\frac{\alpha}{2}},
\qquad
B'':=\cot(\alpha/2),
$$
and
\begin{equation}\label{eq:Xp}
X''_0:=\frac{\sin \theta}{u-\cos \theta},\quad
X''_1:=\frac{\sin (\theta-\alpha)}{u-\cos(\theta-\alpha)},\quad
X''_2:=\log\left(
\frac{u-\cos(\theta-\alpha)}{u-\cos \theta}\right).
\end{equation}
Since
$$
\vmt{-A'' & -A'' & -A''B'' \\
        A''  &   A''\cos \alpha & A''B'' \\
         0  & B''/2  & A''
}
=-(A'')^3(1-\cos \alpha)\ne 0,
$$
and
$$
\pmt{X''_0\\ e^{X''_2}-1}
=
\pmt{0 & 1\\ 1-\cos \alpha & -\sin \alpha}
\pmt{\xi \\ \eta
},
$$
the embeddedness reduces to 
the injectivity of the map $\Psi$
(cf. \eqref{eq:Psi}).
The properness follows from the fact that
$\tilde x_2$ diverges.
In fact, 
either $X''_1$ or $X''_2$ diverges
along the sequence $\{(u_m,\theta_m)\}$ 
in $\Omega_{0,0,\alpha,\alpha}$ converging
to the boundary of $\Omega_{0,0,\alpha,\alpha}$.

\medskip
Finally, we consider the case 
that 
$(\alpha_0,\alpha_1,\alpha_2,\alpha_3)=(0,0,0,\alpha)$.
Then,
$$
\tilde x_0+\tilde x_1=\frac{1}{4\sin(\alpha/2)}X''_2,\quad
\tilde x_2=-\frac{\sin \alpha}{8\sin^3(\alpha/2)}X''_2
-\frac{1}{2\sin(\alpha/2)}X''_0
$$
hold,
where $X''_0$ and $X''_2$
are given in \eqref{eq:Xp}.
So, the assertion  reduces to 
the  injectivity of the map 
$(u,\theta)\mapsto (X''_0,X''_2)$
proved as above.
The properness follows from the fact that
either $X''_0$ or $X''_2$ diverges
along the sequence $\{(u_m,\theta_m)\}$ 
in $\Omega_{0,0,0,\alpha}$ converging
to the boundary of $\Omega_{0,0,0,\alpha}$.
In this case, we can prove that $\tilde x_2$ diverges.

The remaining case 
that 
$(\alpha_0,\alpha_1,\alpha_2,\alpha_3)=(0,0,0,0)$
has been discussed in Example \ref{exa:0000},
\red{and one can easily check that the map $f$ given
in \eqref{eq:f0000} is proper.}
So we get the assertion.
\end{proof}

\appendix
\section{Generalized (anti-)self-reciprocal polynomials}
Let
\begin{equation}\label{eq:A0}
p(r) = a_0 r^n + a_1r^{n-1} +\cdots+ a_{n-1}r + a_n
\end{equation}
be a polynomial in $r$ with complex coefficients.
Then, the polynomial whose coefficients are written 
in reverse order is given by
$$
\check p(r) = a_n r^n + a_{n-1}r^{n-1} +\cdots+ a_{1}r + a_0,
$$
which is called the {\it reciprocal polynomial} of $p(r)$.
If $p(r)$ satisfies
$$
\check p(r)=p(r),
\qquad (\mbox{resp. }\, \check p(r)=-p(r)),
$$
then it is called a {\it self-reciprocal polynomial}
(resp.~an  {\it anti-self-reciprocal polynomial}).
The following assertion can be proved easily.

\begin{lemma}\label{lem:A}
A given polynomial $p(r)$ of degree $n$
is self-reciprocal
$($resp. anti-self-reciprocal$)$
if and only if 
$p(1/r)=p(r)/r^{n}$
$($resp.  $p(1/r)=-p(r)/r^{n})$.
\end{lemma}

Regarding this fact, we give the following 
definition:

\begin{definition}
A given polynomial $p(r)$ is called
a {\it generalized  self-reciprocal polynomial}
(resp. {\it generalized anti-self-reciprocal polynomial})
if there exists a positive integer $m$ such that
$p(1/r)=p(r)/r^{m}$
(resp.  $p(1/r)=-p(r)/r^{m}$).
The number $m$ is called the {\it reciprocal order}
of the polynomial.
\end{definition}

One can easily observe that for each generalized
self-reciprocal polynomial (resp.~generalized 
anti-self-reciprocal polynomial) $Q(r)$,
there exist a self-reciprocal polynomial
(resp. anti-self-reciprocal polynomial) $p(r)$
and a non-negative integer $l$ such that $Q(r)=r^l p(r)$.

A {\it Chebyshev polynomial} $T_n(r)$ 
(resp. $U_{n-1}(r)$) of the {\it first kind}
(resp. the {\it second kind})
is a polynomial of degree $n$ (resp. $n-1$) 
satisfying the following identity
\begin{equation}\label{eq:id}
\frac{r^n+r^{-n}}2=
T_n\left(\frac{r+r^{-1}}2\right)
\quad 
\left(\mbox{resp. } \frac{r^n-r^{-n}}2
=\left(\frac{r-r^{-1}}2\right) U_{n-1}
\left(\frac{r+r^{-1}}2\right) \right).
\end{equation}
The polynomial
$$
p_+(r)=r^n \cdot \frac{r^n+r^{-n}}2\qquad
\left(\mbox{resp. }\,\, 
p_-(r)=r^n \cdot\frac{r^n-r^{-n}}2\right)
$$
is a typical example of a self-reciprocal polynomial
$($resp. anti-self-reciprocal polynomial$)$, 
and the identity \eqref{eq:id}
is the special case of the following assertion: 

\begin{proposition}\label{prop:A}
Let $p(r)$ be a generalized self-reciprocal polynomial
$($resp. generalized anti-self-reciprocal polynomial$)$ 
of reciprocal order $2m$. Then there exists a polynomial $q(u)$ 
such that
$$
p(r)=r^m q\left(
\frac{r+r^{-1}}2
\right),\qquad 
(\mbox{resp. }\, 
p(r)=r^m \left(
\frac{r-r^{-1}}2\right)
q\left(
\frac{r+r^{-1}}2
\right)).
$$
\end{proposition}

\begin{proof}
Let $p(r)$ be a polynomial as in \eqref{eq:A0}.
In particular, $p(r)$ is a polynomial in degree $n$.
Using the property of the Chebyshev polynomial $T_n(r)$
of the first kind,  we have that
\begin{align*}
r^{-m}p(r)
&=
\frac{r^{-m}p(r)+r^{m}p(1/r)}2\\
&=\sum_{j=0}^{n}a_j\frac{r^{n-j-m}+r^{m+j-n}}2 
=\sum_{j=0}^{n}a_j
T_{|n-j-m|}\left(\frac{r+r^{-1}}2\right).
\end{align*}
So the polynomial
$
q(u):=\sum_{j=0}^{n}a_jT_{|n-j-m|}(u)
$
attains the desired property.
The case that $p(r)$ is
a generalized anti-self-reciprocal polynomial of 
reciprocal order $2m$
can also proved in the same manner using the
properties of the Chebyshev polynomial 
of the second kind.

\end{proof}

\end{document}